\documentclass[11pt,leqno]{amsart}

\usepackage{amssymb,upgreek,graphicx,color}  
\usepackage{amsmath}
\usepackage{amsthm}
\begin{document}
\theoremstyle{plain}
\newtheorem{thm}{Theorem}[section]
\newtheorem{prop}[thm]{Proposition}
\newtheorem{lem}[thm]{Lemma}
\newtheorem{clry}[thm]{Corollary}
\newtheorem{deft}[thm]{Definition}
\newtheorem{hyp}{Assumption}
\newtheorem*{KSU}{Theorem (Kenig, Sj\"ostrand and Uhlmann)}
\newtheorem*{helgason}{Helgason's support theorem}
\newtheorem*{microhelgasonholgrem}{Microlocal Helgason-Holmgren Theorem}
\newtheorem*{cohelgason}{Microlocal Helgason's theorem}
\newtheorem*{kashiwara}{Kashiwara's Watermelon theorem}

\theoremstyle{definition}
\newtheorem{rem}[thm]{Remark}
\numberwithin{equation}{section}
\newcommand{\eps}{\varepsilon}
\newcommand{\e}{\mathrm{e}}
\renewcommand{\phi}{\varphi}
\renewcommand{\d}{\partial}
\newcommand{\dd}{\mathrm{d}}
\newcommand{\re}{\mathop{\rm Re} }
\newcommand{\im}{\mathop{\rm Im}}
\newcommand{\R}{\mathbf{R}}
\newcommand{\C}{\mathbf{C}}

\newcommand{\Sph}{\mathbf{S}}
\newcommand{\N}{\mathbf{N}}  
\newcommand{\Rad}{\mathcal{R}} 
\newcommand{\Hilb}{\mathcal{H}}
\newcommand{\SB}{\mathcal{T}_h}
\newcommand{\w}{\mathcal{W}} 
\renewcommand{\div}{\mathop{\rm div}}
\newcommand{\id}{\mathop{\rm Id}}
\newcommand{\D}{\mathcal{C}^{\infty}_0} 
\newcommand{\E}{\mathcal{E}}
\newcommand{\supp}{\mathop{\rm supp}}
\newcommand{\ch}{\mathop{\rm ch}}
\newcommand{\musupp}{\mathop{\rm \upmu supp}\nolimits_{\rm A}}
\newcommand{\WF}{\mathop{\rm WF}\nolimits_{\rm A}}
\newcommand{\bou}{\partial\Omega}
\newcommand{\norm}[3]{\left\|#1\right\|^{#2}_{#3}}  
\newcommand{\duality}[2]{\left\langle #1 \Big| #2 \right\rangle} 
\newcommand{\B}{\mathrm{B}}
\newcommand{\F}{\mathrm{F}}

\newcommand{\tr} {\mathbf{tr}}

\newcommand{\comment}[1]{\marginpar{\footnotesize #1}}

\title[]{Stability estimates for the Radon transform with restricted data and applications}
\author[]{Pedro Caro \and David Dos Santos Ferreira \and Alberto Ruiz }
\address{Department of Mathematics and Statistics, University of Helsinki, Helsinki, Finland}
\email{pedro.p.caro@gmail.com}
\address{Universit\'e Paris 13, Cnrs, Umr 7539 Laga, 99, avenue Jean-Baptiste Cl\'ement, F-93430 Villetaneuse, France}
\email{ddsf@math.univ-paris13.fr}
\address{Departamento de Matem\'aticas, Universidad Aut\'onoma de Madrid, Campus de Cantoblanco, 28049 Madrid, Spain}
\email{alberto.ruiz@uam.es}
\begin{abstract}
   In this article, we prove a stability estimate going from the Radon transform of  a function with limited angle-distance data to 
   the $L^p$ norm  of the function itself, under  some conditions on the support of the function.   We apply this theorem to 
   obtain stability estimates for an inverse boundary value problem with partial data.
   \end{abstract}
\maketitle

\setcounter{tocdepth}{1} 
\tableofcontents
%
%
\begin{section}{Introduction}
In this work we prove  a stability estimate from the  Radon transform with  limited angle-distance data to a local $L^p$-norm of the function. Our original  motivation to study this problem was to obtain stability estimates for the inverse problem in electric impedance tomography (E.I.T.) 
proposed by Calder\'on. Nevertheless, we think that the results obtained on the Radon transform restricted to some partial data sets are interesting 
by themselves and are the main contribution of this  work.
   
Calder\'on's inverse problem  deals with the  recovery  of  a conductivity $\gamma$ in the interior of  a smooth domain $\Omega$
from boundary measurements  realized  by the  Dirichlet-to-Neumann map. Let $u$ be the solution of the Dirichlet boundary value problem 
\begin{align}
\label{Intro:DirichletCond}
     \begin{cases}
           \div(\gamma \nabla u) =0  & \text{ in } \Omega \\ u |_{\d \Omega} = f \in H^{\frac{1}{2}}(\d \Omega)
     \end{cases}
\end{align}
where  $\gamma$ is a positive function of class $\mathcal{C}^2$ on $\bar{\Omega}$.
The Dirichlet-to-Neumann map assigns to a function $f \in H^{\frac{1}{2}}(\d \Omega)$ on the boundary the corresponding
Neumann data of \eqref{Intro:DirichletCond}
     $$ \Lambda_{\gamma} f = \gamma \d_{\nu} u|_{\d\Omega} $$
where $\d_{\nu}$ denotes the exterior normal derivative of $u$.
This is a bounded operator $\Lambda_{\gamma}:H^{\frac{1}{2}}(\d \Omega)
\to H^{-\frac{1}{2}}(\d \Omega)$ --- in fact a pseudodifferential operator of order $1$ when $\gamma$ is smooth. 
The inverse problem formulated by Calder\'on \cite{C} is whether it is possible to determine $\gamma$ from
$\Lambda_{\gamma}$. In fact in its initial formulation, the problem concerns only positive measurable conductivities bounded from above,
and it was solved in dimension $2$ in this degree of generality by Astala and P\"aiv\"arinta \cite{AP}  and remains so far open in higher dimensions.

  This question is related to the inverse problem of determining a bounded potential $q \in L^{\infty}(\Omega)$ in the Schr\"odinger equation 
\begin{align}
\label{Intro:DirichletSchrod}
     \begin{cases}
            -\Delta u + q u=0  & \text{ in } \Omega \\ u|_{\d \Omega} = f \in H^{\frac{1}{2}}(\d \Omega),
     \end{cases}
\end{align}
from boundary  measurements. This reduction was exploited by Sylvester and Uhlmann in \cite{SyU} and in combination with the boundary
determination results on the conductivity obtained by Kohn and Vogelius \cite{KV} allowed them to solve the Calder\'on problem
for smooth conductivities in dimension $n \geq 3$. When $0$ is not a Dirichlet eigenvalue of the Schr\"odinger 
operator $-\Delta +q$, the measurements are implemented by the Dirichlet-to-Neumann map, which can be similarly defined 
as for the conductivity equation by
      $$ \Lambda_{q} f = \d_{\nu} u|_{\d\Omega}. $$
With a slight abuse of notations, we use the convention that whenever the subscript contains the letter $q$, the notation
refers to the Dirichlet-to-Neumann map related to the Schr\"odinger equation \eqref{Intro:DirichletSchrod}, while
if it contains the letter $\gamma$ it refers to the map related to the conductivity equation \eqref{Intro:DirichletCond}.

In the inverse problem with partial data one wonders whether one or the other of the Dirichlet-to-Neumann maps $\Lambda_{\gamma},\Lambda_q$
measured only on a subset of the boundary, determines the conductivity $\gamma$ or the electric potential $q$ inside $\Omega$. 
In dimension two, this problem is settled by the articles \cite{IUY1, IUY2} of Imanuvilov, Uhlmann and Yamamoto
using ideas from Bukhgeim \cite{B} who dealt with  the inverse problem for the Schr\"odinger equation with full data. See \cite{GZ} for this problem on Riemann surfaces.
In dimension higher than three, the first results were obtained by Bukhgeim and Uhlmann  \cite{BU} but 
required measurements on roughly half of the boundary. The results obtained by Kenig, Sj\"ostrand and Uhlmann 
\cite{KSU} are the most precise so far in dimension $n \geq 3$ since they require measurements on small subsets
of the boundary for, say, strictly convex domains $\Omega$. This result has been extended to the Dirac system by Salo and Tzou in \cite{SaTz}. 
We should  also mention the local inverse problem, in which all the measurement are restricted  to input Dirichlet  data supported on the  same (the accesible boundary) subset as    the   output measurements. This problem was settled by Imanuvilov, Uhlmann and Yamamoto \cite{IUY1, IUY2} in dimension $n=2$
and only for very special cases (the complement of the accesible boundary being a piece of a plane or  a sphere)  in dimension $n \geq 3$ by Isakov \cite{Is}
and extended to Maxwell equation in \cite{CaSaOl} and \cite{Ca11}. The linearized inverse Calder\'on problem with partial
data was studied in \cite{DKSU2}.

 Let us describe Bukhgeim and Uhlmann result 
   in more details. For this purpose, given  a direction $\xi \in \Sph^{n-1}$, we    consider  the $\xi$-illuminated  face of $\d\Omega$
\[\partial  \Omega_-(\xi) =\{x\in \partial \Omega: \langle \xi,\nu(x)\rangle \leq 0\}\]
and the $\xi$-shadowed  face
\[\partial  \Omega_+(\xi) =\{x\in \partial \Omega:  \langle \xi,\nu(x)\rangle \geq 0\},\]
where $\nu(x)$ is the exterior normal vector at $ x $.
 \begin{thm}[Bukgheim and Uhlmann \cite{BU}]
      Let $\Omega$ be a bounded open set in $\R^n, n \ge 3$ with smooth boundary 
      and  let us consider $F
      \subset \d\Omega $ 
      an open neighborhood   
      of the face  $\d \Omega_-(\xi)$. 
       Let $q_1,q_2$ be two bounded potentials on $\Omega$, suppose that $0$ is 
      neither a Dirichlet eigenvalue of the Schr\"odinger operator $-\Delta+q_1$ nor of $-\Delta+q_2$, and that for all
      $f  \in H^{1/2}(\d\Omega)$    
       the two Dirichlet-to-Neumann maps coincide on $F$
           $$ \Lambda_{q_1}f {|_{F} }= \Lambda_{q_2}f{|_{F } }, $$         then the two potentials agree $q_1=q_2$.
\end{thm}

 To describe the uniqueness result  of Kenig, Sj\"ostrand and Uhlmann we need to introduce the appropriate parts of $\d \Omega$. 
 Assume $y_0$ is not in the convex hull of $\Omega$, we define the $y_0$-illuminated face as 
 \[\partial  \Omega_-(y_0) =\{x\in \partial \Omega: \langle x-y_0,\nu(x)\rangle \leq 0\}\]
and the $y_0$-shadowed face as
\[\partial  \Omega_+(y_0) =\{x\in \partial \Omega:  \langle x-y_0,\nu(x)\rangle \geq 0\}.\]
 Note the abuse of notation when writing $ \partial \Omega_\pm(\xi) $ and $ \partial \Omega_\pm(y_0) $, since the former one denotes the $ \xi $-illuminated and $ \xi $-shadowed faces from the direction $ \xi $ while the latter one denotes $ \xi $-illuminated and $ \xi $-shadowed faces from the point $ y_0 $. Then

 \begin{thm}[Kenig, Sj\"ostrand and Uhlmann \cite{KSU}]
      Let $\Omega$ be a bounded open set in $\R^n, n \ge 3$ with smooth boundary 
      and  let us consider $F,B\subset \d\Omega $ two open neighborhoods   respectively of the faces  $\d \Omega_-(y_0)$ and $\d  \Omega_+(y_0)$.
       Let $q_1,q_2$ be two bounded potentials on $\Omega$, suppose that $0$ is 
      neither a Dirichlet eigenvalue of the Schr\"odinger operator $-\Delta+q_1$ nor of $-\Delta+q_2$, and that for all
      $f \in H^{1/2}(\d\Omega)$  supported in $B$
       the two Dirichlet-to-Neumann maps coincide on $F$
          $$ \Lambda_{q_1}f|_{F}  = \Lambda_{q_2}f|_{F} , $$  
      then $q_1=q_2$.
\end{thm}

The main goal of this article is to derive an estimate for the Radon transform which  yields the corresponding stability estimates for the above uniqueness results (actually only  a local estimate in the case  of Theorem 1.2).  This estimate, which  we call    a quantitative version of Helgason-Holmgren theorem, will be obtaned in section 2 (see Theorem \ref{th:localLOGstability}).
Stability estimates for the conductivity inverse problem in dimension higher than three go back  to Alessandrini's article \cite{A}. 
This was followed by results in two dimensions by  Liu \cite{L},
Barcelo, Barcelo and Ruiz \cite{BBR}, Barcelo, Faraco and Ruiz \cite{BFR} and finally by Clop, Faraco and Ruiz \cite{CFR} for discontinuous  conductivities corresponding to the uniqueness results of Astala and P\"aiv\"arinta~\cite{AP} (see also \cite{FRo}). Other stability results for the Calder\'on problem in dimension greater than two are \cite{Hk} and \cite{CaGaRe}. In the case of Maxwell equations the stability was obtained in \cite{Ca10}.

Concerning the inverse problem with partial
data, stability estimates corresponding to the results of Bukhgeim and Uhlmann were derived by Heck and Wang \cite{HW},
and in the presence of a magnetic field by Tzou \cite{Tz}. We mention also the uniqueness results obtained by 
Ammari and Uhlmann \cite{AU} in the case where the potential is known close to the boundary, and the corresponding
stability estimates obtained by  Fathallah \cite{Fa}  and  Ben Joud \cite{BJ}. One single $\log$  stability estimate was obtained  by Alessandrini and Kim \cite{AK} in the case of the conductivity equation when the conductivities  coincide on a neighborhood of the boundary with a known one.
The stability of the local problem under similar condition  as in \cite{Is}  was proved 
by Caro for  Maxwell equations in \cite{Ca11}.

Let $F$ and $B$ be boundary neighborhoods of the illuminated and shadowed faces respectively. The natural norm to consider on the partial  Dirichlet-to-Neumann map is 
\begin{multline*} 
   \|\Lambda_{q}\|_{B \to F}= \sup \Big\{ \duality{\Lambda_q (\phi)}{\psi}:\|\phi\|_{ H^{1/2}(\d\Omega)}=\| \psi\|_{ H^{1/2}(\d\Omega)}=1, 
   \\ \supp \phi \subset B,  \supp \psi \subset F \Big\},
\end{multline*} 
 where     $ \langle \cdot | \cdot \rangle $ denotes the duality between $ H^{1/2} (\partial \Omega) $ and $ H^{-1/2} (\partial \Omega) $.
We'll also have to consider a larger norm related to solutions of  Schr\"odinger equation belonging to the space $H( \Omega,\Delta)$ (see section 3.1).
This norm was considered  by Nachman and  Street  in \cite{NS}, where they prove the reconstruction of some 2-plane integrals of the potential from partial  data. 
We will denote this norm as 
     $$\|\Lambda_{q_1} -\Lambda_{q_2} \|^*_{B\to F}$$
The class  of allowable potentials under consideration will be in Besov spaces 
    $$ {\mathcal{K}}(M, \lambda, p) = \big\{q \in L^{\infty}(\Omega)\text{ and }  q{\mathbf {1}}_\Omega \in W^{\lambda, p}(\R^n) : \|q\|_{L^{\infty}}+\|q\|_{W^{\lambda, p} } \leq M \big\} ,$$  
where $\lambda>0$. This class of potential has the advantage of allowing very rough functions  if $\lambda$ is sufficiently small.
Our stability results are as follows.
\begin{thm}\label{BUstability}
      Let $\Omega$ be a bounded open set in $\R^n, n \ge 3$ with smooth boundary.  
       Given  an open set $N$ in $\Sph^{n-1}$  
       consider  $ {F}, {B}\subset \d\Omega$ two open subsets of the boundary which are respective neighbourhoods 
       of the faces $\d \Omega_-(\xi)$ and $ \d\Omega_+(\xi)$ for all directions $\xi \in N$.   
      Given $M>0$ there exists a constant $C>0$ such that the following estimate holds true
     \begin{equation}
            \|q_1-q_2\|_{L^{p} }\leq C\left( \log\big| \log\|\Lambda_{q_1} -\Lambda_{q_2} \|^*_{B\to F} \big|\right)^{-\lambda/2}
      \end{equation}
      for all allowable potentials $q_1,q_2 \in \mathcal{K}(M, \lambda, p)$ on $\Omega$, with $ 1 \leq p < \infty $ and $ 0 < \lambda < 1/p $, for which  $0$ is 
      neither a Dirichlet eigenvalue of the Schr\"odinger operator $-\Delta+q_1$ nor of $-\Delta+q_2$.
\end{thm}

Next we consider the case of illumination from  a point. Let $N$ be  an open set which does not cut the convex hull of $\Omega$.  We will define $P$, the convex penumbra boundary from $N$, as the  set of points $x \in \d\Omega$ such that there exist a $y\in N$ with $\langle x-y, \nu(x)\rangle =0$ and the hyperplane through $x$ normal to $\nu(x)$    being a supporting hyperplane  of $\Omega$. In order to keep the exposition simple, and relate to the Radon transform
(rather than the two-plane transform in high dimensions) we  will  restrict ourselves to the three dimensional case  $n=3$.
 \begin{thm}\label{KSUstability}
      Let $\Omega$ be a bounded open set in $\R^3$ with smooth boundary.  
       Given  an open set $N$ in $\R^{3}$ which does not cut the convex hull of $\Omega$,
      consider two open subsets $ {F}, {B}$ of the the boundary which are respective neighbourhoods  of the  faces $ \partial \Omega_-(y)$ and $\partial \Omega_+(y)$
      for all $y \in N$. Given $M>0$, there exist an open neighborhood $ G \subset \R^{3}$ of the convex penumbra $P$ and a constant $C>0$ such that the following estimate holds true
     \begin{equation}
            \|q_1-q_2\|_{L^p(G) }\leq C\left(\log  \big|\log\|\Lambda_{q_1} -\Lambda_{q_2} \|_{B\to F} \big|\right)^{-\lambda/2}
      \end{equation}
      for all allowable potentials $q_1,q_2 \in \mathcal{K}(M, \lambda, p)$ on $\Omega$, with $ 1 \leq p < \infty $ and $ 0 < \lambda < 1/p $, for which  $0$ is 
      neither a Dirichlet eigenvalue of the Schr\"odinger operator $-\Delta+q_1$ nor of $-\Delta+q_2$.
\end{thm}
The proofs of these theorems will be carried out by using  the   approach of  \cite{DKSU1},  which uses the  Radon transform. 
One can see that the result on the Radon transform is  general enough to be applied to get stability for partial data in the  
context of \cite{KSU} in dimension three (the Dirichlet-to-Neumann  map in this case controls the 2-plane transform, which   
is indeed the Radon transform in three dimensions). This can be achieved with the natural norm   $H^{1/2}(\d\Omega) \to H^{-1/2}(\d\Omega)$ 
of the Dirichlet-to-Neumann map, by using the solutions constructed by Chung in \cite{Ch}. 

Theorem \ref{BUstability} was proved in \cite{HW} by Heck and Wang, without the condition  of the   Dirichlet data being supported on $B$ and the norm in the partial Dirichlet-to-Neumann  map  considered from $H^{3/2}(\partial \Omega)$ to $ H^{1/2}(\partial \Omega) $ instead of the norm $\|\cdot\|^*_{B\to F}$. They  use the Fourier transform. The change to  the Radon transform illustrates the use of Theorem \ref{th:localLOGstability}.

The structure of this paper is as follows. In section \ref{sec:radon}, 
we  will state and prove the  theorem for the  Radon  transform which is  the main result in this  work.
Section \ref{sec:BU} and Section \ref{sec:KSU} are devoted to prove Theorem \ref{BUstability} and Theorem \ref{KSUstability} applying the stability estimates for the Radon transform proven in Section \ref{sec:radon}.

\subsection*{Acknowledgements} 
The authors would like to thank MSRI and the organisers
of the   program on Inverse Problems 2010 for allowing them to benefit from the outstanding research environment in Berkeley.
The project was continued during the Inverse Problems semester at the UAM-ICMAT in Madrid and  at the  Isaac Newton Institute in Cambridge during 2011.
The authors  wish  to thank these institutions for their hospitality.  We would like to thank Giovanni Alessandrini, Jan Boman, Adrian Nachman, Plamen Stefanov and Gunther Uhlmann for conversations  and some  suggestions that improved our original result. PC has been funded by IT-305-07 and 267700 - InvProb along this project.
DDSF's  visit to Madrid was made possible by the project \textit{Inverse Problem and Scattering in Madrid} (CSD2006-00032). He  wishes  to thank Carlos Kenig, Johannes Sj\"ostrand and Gunther Uhlmann for the past discussions on the Watermelon approach.

\end{section} 

\begin{section}{Stability for the  local Radon transform}\label{sec:radon}

We recall that the Radon transform of a continuous compactly supported function $f$ is given by
\begin{align*}
     \Rad f (s,\omega) = \int \delta\big(\langle x,\omega \rangle -s\big) f(x) \, \dd x, \quad s \in \R, \, \omega \in \Sph^{n-1}. 
\end{align*}
  It is always possible to define, by duality,  the Radon transform on compactly supported  distributions.   For the time being, we content ourselves with  continuous functions with some kind of decay, but later on we  will extend its definition  to  a wider class of  functions. This transform is even.
It is sometimes convenient to think of the Radon transform as a function on ${\mathbf  \Pi}^{n-1}$, 
the Grassmannian set of hyperplanes in  $\R^n$. Let  $H$ belong to $ {\bold  \Pi}^{n-1}$, then
\begin{align*}
     \Rad f (H) = \int_{H}  f(x) \, \dd \mu_{H}(x)
\end{align*}
where $\dd\mu_{H}$ is the Lebesgue measure on $H$. To relate both notations in a coherent way, we set
    $$ H_0(s,\omega) = \big\{x \in \R^n :  \langle x,\omega \rangle = s \big\}. $$
Note that $\omega$ is a unit normal to the hyperplane $H_0(s,\omega)$ and $s$ a signed distance to the origin.
 For later convenience we might change the origin of the affine reference for the Radon transform  to  the point  $ y_0 \in \R^n $. If one    describes $ H $ as
\begin{equation}\label{hiperp} H = H_{y_0}(s,\omega)=\{ x \in \R^n : \langle x - y_0,\omega \rangle= s \} 
\end{equation}
for some $ \omega \in \Sph^{n - 1} $ and $ s \in \R $. Relating $ \omega, s $ to $ H $ as above, one can define
\begin{equation}\label{translradon} \Rad_{y_0} f (s, \omega) = \int_H f \, \dd \mu_{H}. 
\end{equation}
We will also make use of the following notation
   $$ H^{\pm}_{y_0}(s,\omega)= \big\{x \in \R^n :  \pm (\langle x-y_0,\omega \rangle-s) < 0 \big\} $$
to denote the half-spaces delimited by $H_{y_0}(s,\omega)$.

We refer to Helgason's book \cite{He} for a general study of the Radon transform.
 Of particular importance is the issue of local inversion of the Radon transform:  given a function $f$  with some a priori regularity  
and some decay at infinity, such that $\Rad f (H) =0$  for every $H \in {\bold\Xi}\subset {\bold  \Pi}^{n-1}$, does $f$ vanish on $E=\cup_{H\in {\bold\Xi}}H$? 
For instance, the celebrated Helgason's support theorem reads as follows.
\begin{helgason}
    Let $f$ be a rapidly   decreasing continuous function such that its Radon transform   vanishes on all hyperplanes disjoint 
    from a compact convex set $K$ 
         $$ \Rad f(H) = 0 , \quad H \cap K = \varnothing $$
    then the support of $f$ is contained in $K$.
\end{helgason}

  We are interested in the microlocal approach
(which differs from Helgason's original proof) to prove Helgason's support theorem presented in \cite{Bo,BoQ,Ho2}. 
This approach is somewhat flexible since it does not require the full family of hyperplanes used in Helgason's theorem 
but can be adapted to provide weaker support results when the Radon transform only vanishes in a neighbourhood of 
a fixed hyperplane. A result that follows from this approach is :
\begin{microhelgasonholgrem} Let $f$ be  a compactly supported continuous function such that  its Radon transform vanishes in  a neighborhood of  $H_0(\langle x_0,\xi_0\rangle,\xi_0)$. If  $\supp f \subset  H^+_0(\langle x_0,\xi_0\rangle,\xi_0)$ then $x_0 \notin \supp f$.
\end{microhelgasonholgrem}
 From the inverse problems point of view, the above result
was used in \cite{DKSU1} to prove the unique  determination of the electric potential and the magnetic field in a
magnetic Schr\"odinger equation from partial data. It served as a substitute to the original but somewhat more involved
argument of Kenig, Sj\"ostrand and Uhlmann in \cite{KSU} also based on analytic microlocal theory. Similar ideas 
were used in \cite{DKSU2} to investigate a linearization of the Calder\'on problem with partial data.

The main result in this section, Theorem \ref{th:localLOGstability}, is  a quantitative version  of  the microlocal Helgason-Holmgren theorem. We want  to  relax the compact support and  continuity assumptions on $f$,  in order to apply the corresponding results   to the study of the  stability of Calder\'on's inverse problem.
In the next paragraphs we  review  some concepts and results that will be basic in the microlocal approach   and that will clarify the proof of Theorem \ref{th:localLOGstability}.

\subsection{Microlocal Helgason's support  and Kashiwara's Watermelon theorems}

We will use the classical notation $w^2=w_{1}^2+\cdots+w_{n}^2$ to denote the holomorphic continuation of the Euclidean
scalar product --- particularly to avoid confusion with the norm $|w|^2=|w_{1}|^2+\cdots+|w_{n}|^2$ of complex vectors. The Segal-Bargmann transform 
of an $L^{\infty}$ function is given by
\begin{align*}
     \SB f(z) = \int \e^{-\frac{1}{2h}(z-y)^2} f(y) \, \dd y, \quad z \in \C^n.
\end{align*}
Note that it has the following exponential growth
\begin{align}
    |\SB f(z)| \leq (2\pi h)^{\frac{n}{2}} \e^{\frac{1}{2h}(\im z)^2}  \|f\|_{L^{\infty}}.
\end{align}
By duality, it is easy to extend this transform to tempered distributions. 
This transform has a wide range of applications in Analysis; amongst others, it provides a way of describing analytic singularities
of a distribution%
\footnote{In fact, this analysis can be extended to hyperfunctions.} on an open set $\Omega \subset \R^n$.
\begin{deft}
     A distribution $f \in \mathcal{D}'(\Omega)$ is said to be microlocally exponentially small at $(x_{0},\xi_{0}) \in T^*\Omega$ if there 
     exist a cutoff function $\chi \in \D(\Omega)$ such that $\chi(x_{0}) \neq 0$, two constants $c,C>0$ and a neighbourhood $V_{z_{0}}$
     of $z_{0}=x_{0}-i\xi_{0}$ in $\C^n$ such that the following improved bound holds on the Segal-Bargmann transform
     \begin{align}
         |\SB(\chi f)(z)| \leq C \e^{-\frac{c}{h}+\frac{1}{2h}(\im z)^2}.
     \end{align}
     for all $z \in V_{z_{0}}$ and all $h \in (0,1]$.

     The analytic microsupport of a distribution $f$ --- which we denote by $\musupp f$ --- is the complement of the set of covectors 
     $(x_{0},\xi_{0}) \in T^*\Omega$ at which $f$ is microlocally exponentially small. The analytic wave front set $\WF f$ of $f$
     is the complement  in $T^*\Omega \setminus 0$  of the set of covectors at which $f$ is microlocally exponentially small.
\end{deft}
The analytic microsupport is a closed conic%
\footnote{That is, in the frequency variable: $(x,\xi) \in \musupp f \Rightarrow (x,\lambda \xi) \in \musupp f$ for positive~$\lambda$.} 
 set of the cotangent bundle and consists in two parts
     $$ \musupp f = \supp f \times \{0\} \cup \WF f. $$ 
The projection with respect to the space variable of the analytic wave front set is the analytic singular support
     $$ \pi(\WF f) = \supp\nolimits_{\rm A} f, \quad \pi : T^*\Omega \to \Omega $$
i.e. the set of points $x_{0} \in \R^n$ which have no neighbourhood on which $f$ is real analytic.
A microlocal form of Helgason's support theorem reads as follows.
\begin{cohelgason}
    If the Radon transform  $\Rad f(s,\omega)$ of $f\in {\mathcal C}^0(\R^n)$    vanishes in the neighbourhood of $(s_{0},\omega_{0}) \in \R \times \Sph^{n-1}$
    then $(x_{0},\omega_{0}) \notin \WF f$ where $x_{0} \in H_0(s_{0},\omega_{0})$.
\end{cohelgason}
A more invariant formulation would be that the conormal $N^*(H_0(s_{0},\omega_{0}))$ to the hyperplane is contained in the complement of the analytic wave front set.
As we will see it implies a local weaker (but quite flexible) form of Helgason's support theorem.

In the situation where a distribution is supported on one side of a hyperplane, Kashiwara's Watermelon theorem describes 
some of the covectors of the analytic microsupport (\cite{K}, \cite[Theorem 8.3.3]{Sj}, \cite[Theorem 9.6.6]{Ho1}, \cite{Sj2}).
\begin{kashiwara}
    Let $f \in \mathcal{D}'(\R^n)$ be a distribution supported on one side $H^+$ of a hyperplane $H$. Let $\nu_{0}$ denote a unit normal to $H$.
    If $(x_{0},\xi_{0}) \in \musupp f$ then so does $(x_{0},\xi_{0}+t\nu_{0})$ for all $t \in \R$.
\end{kashiwara}
Kashiwara's Watermelon theorem is generally stated in terms of the analytic wave front set, we chose to use a formulation involving the
analytic microsupport (for similar formulations, see also \cite{M}) since it encompasses information about the support of the function.  
In particular, it immediatly implies the following unique continuation property
     $$ \supp f \subset \big\{\langle x-x_{0},\nu_{0} \rangle <0\big\} \text{ and }Ê(x_{0},\nu_{0}) \notin \WF(f) \Rightarrow x_{0} \notin \supp  f. $$
This is sometimes known as Holmgren's microlocal uniqueness theorem\footnote{In the usual formulation using the analytic wave front set, this is obtained by a limiting argument from the Watermelon theorem.
If $x_{0} \in \supp f$ then $f$ cannot be analytic at $x_{0}$ hence there exists $(x_{0},\xi_{0}) \in \WF f$ therefore $(x_{0},\xi_{0}+t\nu_{0})
\in \WF f$. By the conicity of the wave front set, we have $(x_{0},\xi_{0}/t+\nu_{0}) \in \WF f$ and by the closedness, we have $(x_{0},\nu_{0}) \in \WF f$.}
(or the co-Holmgren theorem).

\subsection{Relating the Radon and the Segal-Bargman transforms}

In this paragraph, we want to connect the Radon and the Segal-Bargman transforms. We start from the identity
\begin{align*}
     \widehat{f}(\sigma \omega) = \widehat{\Rad f}(\sigma,\omega) 
\end{align*}
where $\widehat{\Rad f}(\sigma,\omega) $ denotes the (one-dimensional) Fourier transform of $\Rad f(s,\omega)$ with respect to $s$
    $$ \widehat{\Rad f}(\sigma,\omega)  = \int_{-\infty}^{\infty} \e^{-i s \sigma} \Rad f(s,\omega) \, \dd s $$ 
and use Plancherel's identity to compute the scalar product
\begin{align*}
     \int f \, \overline{g} \, \dd x 
     =  (2 \pi)^{-n} \int_0^{\infty} \int_{\Sph^{n-1}} \widehat{\Rad f}(\sigma,\omega) \, \overline{\widehat{\Rad g}(\sigma,\omega)} \, \sigma^{n-1}
     \, \dd \sigma \, \dd \omega.
\end{align*}
Using the fact that the Radon transform is even, and once again Plancherel's identity, we get 
\begin{align}
\label{Radon:IdRad}
     \int f \, \overline{g} \, \dd x 
     =  \frac{1}{2} (2 \pi)^{-n+1} \int_{-\infty}^{\infty} \int_{\Sph^{n-1}} \Rad f(s,\omega) \, \overline{|D|^{n-1}\Rad g(\cdot,\omega)(s)}
      \, \dd s\, \dd \omega.
\end{align}
We choose $g$ to be the conjugate of the Gaussian kernel of the Segal-Bargman transform: we begin by computing its Radon transform
\begin{align*}
     \Rad\big(\e^{-\frac{1}{2h}(\bar{z}-x)^2}\big)(s,\omega) &= \int \delta\big(\langle x,\omega \rangle -s\big) 
     \e^{-\frac{1}{2h}(\bar{z}-x)^2} \, \dd x \\
     &= (2\pi h)^{\frac{n-1}{2}} \e^{-\frac{1}{2h}(s-\langle \omega,\bar{z}\rangle)^2} 
\end{align*}
and plug $g=\e^{-\frac{1}{2h}(\bar{z}-x)^2}$ in the identity \eqref{Radon:IdRad} to compute the Segal-Bargman transform 
of a function in terms of the Radon transform
\begin{align}
\label{Radon:FBIRad}
     \SB f(z) =  \frac{1}{2} (2 \pi)^{-\frac{n-1}{2}} h^{\frac{n-1}{2}} 
     \int_{-\infty}^{\infty} \int_{\Sph^{n-1}}  \, G_n\big(s,\langle \omega,z \rangle \big) \Rad f(s,\omega)
      \, \dd s\, \dd \omega
\end{align}
where the kernel $G_n$ is given by
   $$ G_n(s,w) = |D|^{n-1}\big(\e^{-\frac{1}{2h}(\cdot-w)^2}\big)(s), \quad s\in \R, \quad w \in \C. $$
   
We will use the following estimates of the kernel:
 \begin{lem}
\label{Radon:EstGnLem} 
     The kernel $G_n$ satisfies the following bound
     \begin{align}
          |G_n(s,w)| \leq B_n h^{-\frac{n-1}{2}} \big(1+h^{-\frac{1}{2}}|s-w|\big)^{n} \big(1+\e^{\frac{1}{2h}((\im w)^2-(s-\re w)^2)}\big).
     \end{align}
\end{lem}
\begin{proof}
We need to distinguish two cases according to the parity of the dimension~$n$.

Let us start with the case $n$ odd which is simpler. The kernel $G_n$ can explicitly be computed 
\begin{align*}
      G_n(s,w) = D_s^{n-1}\big(\e^{-\frac{1}{2h}(s-w)^2}\big) = h^{-\frac{n-1}{2}}
     \e^{-\frac{1}{2h}(s-w)^2} Q_{n}\bigg(\frac{s-w}{\sqrt{h}}\bigg)
\end{align*} 
where 
\begin{align*}
     Q_{n}(w) =  \e^{\frac{w^2}{2}}D_w^{n-1} \e^{-\frac{w^2}{2}}
\end{align*}
is a Hermite polynomial of degree $n-1$, hence satisfies the bound
\begin{align*}
     |Q_{n}(w)| \leq A_{n} (1+|w|)^{n-1}.
\end{align*}
The former estimate together with 
 \begin{align*}
           \big|\e^{-\frac{1}{2h}(s-w)^2}\big| &= 
           \e^{-\frac{1}{2h}(s-\re w)^2+\frac{1}{2h}(\im w)^2}
\end{align*}
imply the following bound on $G_n$
\begin{align}
     |G_n(s,w)| \leq A_n h^{-\frac{n-1}{2}} \big(1+h^{-\frac{1}{2}}|s-w|\big)^{n-1} \e^{\frac{1}{2h}((\im w)^2-(s-\re w)^2)}
\end{align}
when $n$ is odd. 

Notice that identity  \eqref{Radon:FBIRad} reads in odd dimensions 
\begin{multline*}
     \SB f(z) = \\  \frac{1}{2} (2 \pi)^{-\frac{n-1}{2}} 
     \int_{-\infty}^{\infty} \int_{S^{n-1}} \e^{-\frac{1}{2h}(s-\langle \omega,z\rangle)^2}
     Q_{n}\bigg(\frac{s-\langle \omega,z\rangle}{\sqrt{h}}\bigg)\Rad f(s,\omega) \, \dd \omega \, \dd s.
\end{multline*}

 The even dimensional case is a bit more involved: the kernel $G_n$ satisfies the following relations
\begin{align}
\label{Radon:KernelRel}
     G_n(s,w) &= G_n(s-\re w,i \im w), \\ \nonumber \overline{G_n(s,w)} &= G_n(s,\bar{w}) = G_n(-s,-\bar{w})
\end{align}
and has the following expression
\begin{align*}
     G_n(s,w) &= |D |^{n-1}\big(\e^{-\frac{1}{2h}(\cdot-w)^2}\big)(s) \\
     &= h^{-\frac{n-1}{2}} \frac{\e^{-\frac{w^2}{2h}}}{\sqrt{2\pi}}  \int_{-\infty}^{\infty}
     |\sigma|^{n-1} \e^{-\frac{(\sigma+iw/\sqrt{h})^2}{2}+\frac{i}{\sqrt{h}} s \sigma} \, \dd \sigma.
\end{align*}

     By the relations \eqref{Radon:KernelRel}, it suffices to prove the estimate when $w=-i\tau$ is imaginary
     and $s$ is non-negative, and by scaling, we might as well assume that $h=1$. For $\tau \in \R, s \geq 0$ and $n$ even, we
     may decompose the kernel as 
     \begin{align*}
          G_n(s,-i\tau) &=  \frac{\e^{\frac{\tau^2}{2}-is\tau}}{\sqrt{2\pi}} \bigg(\int_{-\infty}^{\tau} + \int_{\tau}^{\infty} 
          |\sigma -\tau|^{n-1} \e^{-\frac{\sigma^2}{2}+i s \sigma} \, \dd \sigma \bigg) \\
          &= I_n(s,\tau) + \bar{I}_n(s,-\tau)
     \end{align*} 
     where the integral
     \begin{align*}
          I_n(s,\tau) = \frac{\e^{\frac{\tau^2}{2}-is\tau}}{\sqrt{2\pi}}  \int_{-\infty}^{\tau} (\tau-\sigma)^{n-1} \e^{-\frac{\sigma^2}{2}+i s \sigma} 
          \, \dd \sigma   
     \end{align*}
     can be computed by an integration on the contour $(-\infty,\tau] \cup [\tau,\tau+is] \cup [\tau+is,-\infty+is)$
     \begin{multline*}
            I_n(s,\tau) = \frac{\e^{\frac{\tau^2}{2}}}{\sqrt{2\pi}}\bigg(
           \int_{-\infty}^{\tau} (\tau-\sigma-is)^{n-1} \e^{-\frac{(\sigma+is)^2}{2}+i s (\sigma+is)} \, \dd \sigma  \\
           + (-i)^{n}\int_0^s \sigma^{n-1} \e^{-\frac{1}{2}(\tau+i\sigma)^2} \e^{is(\tau+i\sigma)} \,\dd \sigma \bigg).
     \end{multline*}
    The first term is bounded by a constant times
         $$ \e^{\frac{\tau^2}{2}-\frac{s^2}{2}} (1+|s|+|\tau|)^{n-1} $$
    while the second term%
    \footnote{In the odd dimensional case, this term disappears when one computes $I_n(s,\tau) + \bar{I}_n(s,-\tau)$.}  
    is bounded by a constant times
        $$  (1+|s|)^{n}. $$
     This completes the proof of the lemma.
\end{proof}

We restate (\ref{Radon:FBIRad}) in term of the Radon transform centered  at $y_0$, see \eqref{translradon},  as
\begin{multline}
\label{Radon:FBIRady}
     \SB f(\zeta) =  \\ \frac{1}{2} (2 \pi)^{-\frac{n-1}{2}} h^{\frac{n-1}{2}} 
        \int_{-\infty}^{\infty} \int_{\Sph^{n-1}}  \, G_n\big(s,\langle \omega,\zeta-y_0 \rangle \big) \Rad_{y_0} f(s,\omega)
      \, \dd s\, \dd \omega.
\end{multline}
Given $ \omega_0 \in \Sph^{n - 1}$  and $ \beta \in (0, 1] $ we consider the following cap centered around $\omega_0$ 
on the hypersphere $\Sph^{n-1}$
\begin{align}
 \label {gamma}
      \begin{aligned}
      \Gamma &= \{ \omega \in \Sph^{n - 1} : \langle\omega , \omega_0\rangle^2 > 1 - \beta^2 \} \\
      &=  \{ \omega \in \Sph^{n - 1} : d_{\Sph^{n-1}}(\omega_0, \omega)< \arcsin \beta\}
      \end{aligned}
 \end{align} 
$ d_{\Sph^{n-1}}$ being the geodesic distance on $\mathbf S^{n-1}$.

   Before proceeding to further computations, we also note that $\dd \mu_{H_0(s,\omega)} \wedge \dd s = \dd x$ 
and therefore
\begin{align*}
     \int_{-\infty}^\infty \big|\Rad f(s,\omega)\big| \, \dd s \leq \|f\|_{L^1}
\end{align*}
which leads to
\begin{align}
\label{Radon:L1Rad}
     \int_{-\infty}^\infty \int_{\Sph^{n-1}} \big|\Rad f(s,\omega)\big|  \, \dd \omega \, \dd s \leq |\Sph^{n-1}| \times \|f\|_{L^1}.
\end{align}

 We   introduce the following set of functions: $ u \in X $ by definition if and only if $ u \in L^1 (\R^n) $ and
\begin{equation}\label{normX}
\|u\|_X= \int_\R (1 + |s|)^n \norm{\Rad_0 u(s, \cdot)}{}{L^1(\Sph^{n - 1})} \, \dd s < \infty.
\end{equation}

Let us remark, see (\ref{Radon:L1Rad}),  that for functions in $L^1 (\R^n) $, the Radon transform       is defined a.e. as a function  in $ L^1(\R \times \Sph^{n - 1}) $. 
  Our space $X$ is more restrictive,  a sufficient condition for  a function to be in $X$, is given by the estimate
\begin{equation}\label{decay}
 \norm{u}{}{X} \leq |\Sph^{n - 1}| \int_{\R^n} (1 + |x|)^n |u(x)| \, \dd x. 
\end{equation}

 \begin{prop}[Quantitative  Microlocal   Helgason's theorem]\label{prop:SB-Rtranforms}
Let $ f $ belong to $ X $.  There exists a positive constant $ C $, only depending on $ n $, such that 
\begin{multline}\label{mhelgason}
  \e^{-\frac{1}{2h}|\mathrm{Im}\, \zeta|^2}|\SB f (\zeta)| \leq  \frac{C}{h^\frac{n}{2}} (1 + |  \zeta |  +|y_0|)^n  \\
  \times  \bigg( \int_{|s|<\alpha}(1 + |s|)^n \norm{\Rad_{y_0}f(s, \cdot)}{}{L^1(\Gamma)} \, \dd s 
  + \norm{f}{}{X} \Big(\e^{-\frac{\alpha^2}{8h}} + \e^{-\frac{\gamma^2 \beta^2}{32 h}}\Big) \bigg),
\end{multline}
for all $ h \in (0, 1] $, $ \gamma > 0 $ and $ \zeta \in \C^n $ such that $ |\mathrm{Re}\, \zeta - y_0| < \alpha / 2 $, $ |\mathrm{Im}\, \zeta| \geq \gamma $ and   $  \langle\omega_0 ,\frac{\mathrm{Im}\,\zeta}{|\mathrm{Im}\,\zeta|}\rangle^2 > 1 - \beta^2 / 4 $ .
\end{prop} 
\begin{rem}
    Note that when $\Rad f$ vanishes on a neighborhood of $(s_0,\omega_0)$ the above estimate  implies that $(y _{0},\pm \omega_{0}) \notin \WF(f)$ 
    when $\langle y_{0},\omega_{0}\rangle=s_{0}$. This is the microlocal version of Helgason's support theorem as stated in
    the introduction of this section. Proposition \ref{prop:SB-Rtranforms} is therefore a quantitative version of this microlocal result.
\end{rem}
\begin{proof} 
It follows from Lemma \ref{Radon:EstGnLem} and \eqref{Radon:FBIRady} that
\begin{multline*}
|\SB f(\zeta)| \leq C \int_{-\infty}^{\infty}\int_{\Sph^{n-1}}  \big(1+h^{-\frac{1}{2}}|s-\langle \omega,\zeta-y_0 \rangle|\big)^{n} 
\\ \times \big(1+\e^{\frac{1}{2h}(\langle\omega,\im \zeta\rangle ^2-(s-\langle\omega, \re \zeta-y_0\rangle)^2)}\big) \,
|\Rad_{y_0} f(s, \omega)|\,\dd \omega\,\dd s.
\end{multline*}
Let us split  the integral  into 
 \[ \int_{\R \times \Sph^{n - 1}} = \int_{\R \setminus (-\alpha, \alpha) \times \Sph^{n - 1}}+\int_{(-\alpha, \alpha) \times (\Sph^{n - 1} \setminus \Gamma)} +\int_{(-\alpha, \alpha) \times \Gamma}= I_1+ I_2+ I_3. \] 
To estimate $I_1$, notice that 
if $ s \in \R \setminus (-\alpha, \alpha) $ and $ \omega \in \Sph^{n - 1} $, then
\begin{align*}
\e^{-\frac{1}{2h}|\langle\omega , \mathrm{Re}\, \zeta - y_0\rangle - s|^2} \leq \e^{-\frac{1}{2h}(|s| - |\langle\omega  , 
\mathrm{Re}\, \zeta - y_0\rangle|)^2} \leq \e^{-\frac{1}{2h}\frac{\alpha^2}{4}},   
\end{align*} and that 
\[1 \leq \e^{-\frac{1}{2h}\gamma^2}\e^{\frac{1}{2h}|\mathrm{Im}\, \zeta|^2},\]
for all $ \zeta \in \C^n $ such that $ |\mathrm{Re}\, \zeta - y_0| < \alpha / 2 $ and $ |\mathrm{Im}\, \zeta| \geq \gamma $. 
Then it follows  easily that 
\[I_1\leq  \frac{C}{h^\frac{n}{2}}\e^{\frac{1}{2h}|\mathrm{Im}\, \zeta|^2}
  (1 + |  \zeta|  )^n \\   
  \norm{f}{}{X} \big(\e^{-\frac{1}{2h}\frac{\alpha^2}{4}} + \e^{-\frac{1}{2h} \frac{\gamma^2}{16}}\big) .
\]
The integral $I_2$ can be bounded in a similar way. Notice  that, for $ s \in (-\alpha, \alpha) $ and $ \omega \in \Sph^{n - 1} \setminus \Gamma $, it holds with $ \theta = |\mathrm{Im}\, \zeta|^{-1} \mathrm{Im}\, \zeta $
\begin{align*}
\e^{\frac{1}{2h}\langle \omega ,\mathrm{Im}\, \zeta\rangle^2} &\leq \e^{\frac{1}{2h}|\mathrm{Im}\, \zeta|^2} \e^{-\frac{1}{2h}(|\mathrm{Im}\, \zeta|^2 - \langle \omega , \mathrm{Im}\, \zeta\rangle^2)} \leq \e^{\frac{1}{2h}|\mathrm{Im}\, \zeta|^2} \e^{-\frac{1}{2h} \gamma^2 (1 - \langle\omega ,\theta\rangle^2)},
 \end{align*} 
 and  again
 \[1 \leq \e^{-\frac{1}{2h}\gamma^2}\e^{\frac{1}{2h}|\mathrm{Im}\, \zeta|^2},\]
for all $ \zeta \in \C^n $ such that $ |\mathrm{Im}\, \zeta| \geq \gamma $ and $ \langle\omega_0 , \theta\rangle^2 > 1 - \beta^2 / 4 $.
This completes the proof of the Proposition.
 \end{proof}

\subsection{A quantitative Helgason-Holmgren theorem}
Now  we  state and prove the main  result in this  section, the quantitative version of microlocal Helgason-Holmgren theorem:
\begin{thm}
\label{th:localLOGstability} Let $ M\geq 1 $ be constant. Given $ y_0 \in \R^n $, $ \omega_0 \in \Sph^{n - 1} $, $ \alpha > 0 $ and $ \beta \in (0, 1] $, consider  
$(-\alpha,\alpha)\times \Gamma \subset \R\times \Sph^{n-1}$ introduced in \eqref{gamma} above, and 
define  the \emph{dependence domain} of the Radon transform data
  \[
E = \big\{ x \in \R^n : \langle \omega , x - y_0\rangle = s,\, s \in (-\alpha, \alpha),\, \omega \in \Gamma \big\}.
\]
Assume that  for some $p$, $ 1 \leq p < \infty $,  and $\lambda$,  $ 0 < \lambda < 1/p $,   a function $ q $ satisfies  the following conditions:
\begin{itemize}
\item[(a)] $ \mathbf{1}_E q  \in X \cap L^\infty(\R^n) $  , where $ \mathbf{1}_E $ stands for the characteristic function of the set $ E $, furthermore $$ \norm{q}{}{L^\infty(E)} + \norm{\mathbf{1}_E q}{}{X} < M .$$
\item[(b)] $ y_0 \in \mathrm{supp}\, q $ and $ \mathrm{supp}\, q \subset \{ x \in \R^n : \langle x - y_0,  \omega_0\rangle \leq 0 \} $.
\item[(c)] $(\lambda,p,p)$-Besov regularity on the dependence domain
\[ \int_{\R^n} \frac{\norm{\mathbf{1}_E q - (\mathbf{1}_E q)(\cdot - y)}{p}{L^p(\R^n)}}{|y|^{n + \lambda p}} \, \dd y < M^p. \]
\end{itemize}
Then there exists a positive constant $ C =C(M,|G|,\alpha, \beta, \lambda)$,  such that 
\begin{equation} \label{es:LOGstabiltyRADON}
\norm{q}{}{L^p(G)} \leq C\bigg|\log \displaystyle\int_{(-\alpha, \alpha)} (1 + |s|)^n \norm{\Rad_{y_0}q(s, \cdot)}{}{L^1(\Gamma)} \, \dd s \bigg| 
^{-\frac{\lambda}{2}},
\end{equation}
where 
\begin{equation}\label{abierto}G = \left\{ x \in \R^n : |x - y_0| < \frac{\alpha}{8 \cosh  ( 8 \pi / \beta )} \right\}.
\end{equation}
\end{thm}
 A precise  value of  $C$ in (\ref{es:LOGstabiltyRADON}) can be given as 
\begin{equation} \label{constantC} C={ C_n M \max \left( 1,\,   |G|^\frac{1}{p} \right) (1 + |y_0|) \left(\alpha^{-n} + \beta^{-n} + \alpha^\lambda\right)}.
\end{equation}

 Compared to Helgason support theorem, we relax  the decay condition to the one given in (\ref {decay})  and  the $L^\infty$ moduli of continuity are relaxed  to integral moduli of continuity, which, under the condition $ 0 < \lambda < 1/p $, {allow} non continuous functions and are preserved (modulo constants)   by multiplication by rough characteristic functions. These facts will be important in the applications.

It will be convenient to use the classical  
 $$  t_{+} = \max (t,0) \quad t_{-} = \min (t,0) $$
to denote the positive and negative parts of a  real number (or a function).
The bounds on the Segal-Bargmann transform can be improved whenever the function $f$ is supported on one side of a hyperplane.
Indeed if 
   $$ \supp f \subset H_0^+(s,\omega_0) $$
then we have, for any $y_0\in H_0(s,\omega_0)$,
\begin{align}\label{condicionsoporte}
   |\SB f(\zeta)| \leq (2\pi h)^{\frac{n}{2}} \e^{\frac{1}{2h}(\im \zeta)^2-\frac{1}{2h}\langle \re \zeta-y_0, \omega_{0}\rangle_{+}^2}  \|f\|_{L^{\infty}}.
\end{align}
Our first step  will be to extrapolate estimate (\ref{mhelgason}) to capture  points $\zeta\in \C^n$  with $\im\zeta =0$. Note that for those values
of the parameter $\zeta$, the Segal-Bargmann transform is a gaussian transform. As in  \cite{Sj}  (see also \cite[Lemma 9.6.5]{Ho1}) the following maximum principle for subharmonic functions will be  the keystone of the proof. 
 
\begin{lem}\label{lem:sub-Harmonic} Let $ a, b $  and $ \lambda$ be positive constants. Consider
\[ R = \big \{ z \in \C : |\mathrm{Re}\, z| < a,\, |\mathrm{Im}\, z| < b + \varepsilon \big\}, \]
for some $ \varepsilon > 0 $. Let $ F $ be a subharmonic function on $ R $ such that 
$$ F(z) < (\mathrm{Re}\, z_{-})^2 ,$$ for all $ z \in R $ and $$ F(z) < -\lambda ,$$ for $ z \in R $ such that $ |\mathrm{Im}\, z| \geq b $. 
Then, for 
\[ |\mathrm{Im}\, z| < b, \qquad |\mathrm{Re}\, z| < \frac{\delta}{2}, \] 
we  have
\[ F(z) < - \frac{\lambda}{2a\cosh \left(\pi \frac{b}{a} \right)} \delta, \]
where \[ \delta =\min \left( \frac{ 
\lambda}{2a\cosh \left( \pi \frac{b}{a} \right)}, \frac{a}{3} \right).\]
\end{lem}
\begin{proof}
     The claim  follows  by comparison of the subharmonic function $F(z) -\delta^2$ with the harmonic function  
         $$ G(z)= -\lambda \frac{\cosh \left( \frac{\pi}{a} y \right)}{\cosh \left( \frac{\pi}{a} b \right)} \sin \left( \frac{\pi}{a} (x + \delta) \right), $$
where $z=x+iy$, is in the   rectangle $R_\delta=[-\delta, a-\delta]\times [-b, b]$. 
 In fact, on the boundary of $R_\delta$ we have
 \begin{align*}
    F(x \pm ib) -\delta ^2&< -\lambda \leq G(x \pm ib) \text{ for } x \in [- \delta, a - \delta ] , \\
    F(- \delta + iy)-\delta ^2 &< 0 =G(- \delta + iy) 
\intertext{and }
     F(a- \delta + iy)-\delta ^2 &< 0 =G(a- \delta + iy).
\end{align*}
From the maximum principle     $ F(z) < G(z)+ \delta ^2$ in $ R_\delta $, which means that 
\begin{align*}
F(x + iy) &< \delta^2 - \lambda \frac{\cosh \left( \frac{\pi}{a} y \right)}{\cosh \left( \frac{\pi}{a} b \right)} \sin \left( \frac{\pi}{a} (x + \delta) \right).
\intertext{Since $ \sin t > 2t / \pi $ for $ 0 < t < \pi / 2 $, one has that}
F(x + iy) &< \delta^2 - \lambda \frac{\cosh \left( \frac{\pi}{a} y \right)}{\cosh \left( \frac{\pi}{a} b \right)} \frac{2}{a} (x + \delta), 
\intertext{whenever $ 0 < x + \delta < a / 2 $. So if $ x $ is restricted to $ |x| < \delta / 2 $, then}
     F(x + iy) &< \delta^2 - \frac{\lambda}{a\cosh \left( \frac{\pi}{a} b \right)}  \delta \\
    &\leq - \frac{\lambda}{2a\cosh \left( \frac{\pi}{a} b \right)}  \delta.
\end{align*} 
This completes the proof of the Lemma.
\end{proof}

\begin{prop} \label{prop:SBnear(y_0,0)} Consider $ q \in X \cap L^\infty(\R^n) $ and let $ y_0 \in \R^n $ and $ \omega_0 \in \Sph^{n - 1} $ be such that $ y_0 \in \mathrm{supp}\, q $ and $ \mathrm{supp}\, q \subset \{ x \in \R^n : \langle x - y_0, \omega_0\rangle \leq 0 \} $. Given $ \alpha > 0 $ and $ \beta \in (0, 1] $ consider the set
\[ \Gamma = \{ \omega \in \Sph^{n - 1} : \langle\omega ,\omega_0\rangle^2 > 1 - \beta^2 \}. \]
If one has
\[ \int_{(-\alpha,\alpha)} (1 + |s|)^n \norm{\Rad_{y_0}q(s, \cdot)}{}{L^1(\Gamma)} \, \dd s \leq \e^{-\frac{\alpha^2}{8}}, \]
there exists a positive constant $ C $, only depending on $ n $, such that 
\begin{multline}\label{es:SBnear(y_0,0)}
\e^{- \frac{1}{2h}|\mathrm{Im}\, \zeta|^2} |\SB q (\zeta)| \\ \leq C M_q  \left(1 + |y_0| + \frac{\alpha}{\beta} \right)^n\, \left( \int_{-\alpha}^\alpha (1 + |s|)^n \norm{\Rad_{y_0}q(s, \cdot)}{}{L^1(\Gamma)} \, \dd s \right)^\kappa,
\end{multline}
with
\begin{gather*}
\kappa < \frac{1}{8 \left( \cosh  \left( \frac{8 \pi}{\beta} \right) \right)^2}, \quad M_q := \max (1, \norm{q}{}{L^\infty(\R^n)} + \norm{q}{}{X}), \\
\quad h = \frac{\alpha^2}{8|\log \int_{-\alpha}^\alpha (1 + |s|)^n \norm{\Rad_{y_0}q(s, \cdot)}{}{L^1(\Gamma)} \, \dd s|},
\end{gather*}
for all $ \zeta \in \C^n $ such that
\begin{equation}\label{hipotesisz} |\mathrm{Re}\, \zeta - y_0| < \frac{\alpha}{8 \cosh  ( 8 \pi / \beta )}, \qquad |\mathrm{Im}\, \zeta| < \frac{2 \alpha}{(4 - \beta^2)^{1/2}}. 
\end{equation}
\end{prop}

\begin{proof}  Let $ \zeta \in \C^n $ and   denote $ z = \langle\omega_0 , \zeta - y_0 \rangle\in \C$. 
We write $ \zeta = (z + \langle\omega_0 , y_0\rangle) \omega_0 + w $ with $ w \in \C^n $ such that $\langle \mathrm{Re}\, w , \omega_0 \rangle= \langle \mathrm{Im}\, w , \omega_0\rangle = 0$. Let us denote
  $$\mathcal  I=\int_{(-\alpha, \alpha)} (1 + |s|)^n \norm{\Rad_{y_0}q(s, \cdot)}{}{L^1(\Gamma)} \, \dd s.$$
Choose $ \gamma =\frac{2\alpha}{\beta}> 0 $ as in Proposition \ref{prop:SB-Rtranforms}, then
\begin{align*}
|\SB q ((z &+ \langle\omega_0, y_0\rangle) \omega_0 + w)| \\ 
&\leq C M_q (1 + \rho+|y_0|)^n\, \e^{\frac{1}{2h}|\mathrm{Im}\, z|^2+\frac{1}{2h}|\mathrm{Im}\, w|^2}  
 \Big( h^{-\frac{n}{2}}  \mathcal I + h^{-\frac{n}{2}} \e^{-\frac{1}{2h}\frac{\alpha^2}{4}} \Big)\\
&\begin{aligned} \leq C M_q (1 + \rho+|y_0|)^n\, &\e^{\frac{1}{2h}|\mathrm{Im}\, z|^2+\frac{1}{2h}|\mathrm{Im}\, w|^2} \e^{-\frac{1}{2h}\frac{\alpha^2}{16}} \\ &\times \Big( h^{-\frac{n}{2}} \e^{\frac{1}{2h}\frac{\alpha^2}{16}}  \mathcal I + h^{-\frac{n}{2}} \e^{-\frac{1}{2h}\frac{3\alpha^2}{16}} \Big) \end{aligned}
\end{align*}
for all $ z \in \C $ and $ w \in \C^n $ such that
\begin{align}
    \label{condicionw1} |\mathrm{Re}\, z|^2 + |\mathrm{Re}\, w - y_0 + \langle \omega_0 , y_0\rangle\omega_0|^2 &< \alpha^2 / 4 , \\
    \label{condicionw2} |\mathrm{Im}\, z| &\geq 2 \alpha / \beta , \\
    \label{condicionw3} |z + \langle \omega_0, y_0 \rangle|^ 2 + |w|^ 2 &< \rho^2 ,  \\
\intertext{and } 
  \label{condicionw4} |\mathrm{Im}\, z|^ 2 / (|\mathrm{Im}\, z|^ 2 + |\mathrm{Im}\, w|^ 2) &> 1 - \beta^ 2 / 4,
\end{align}  
  where $ \rho>0 $ is large enough.
We next consider $ w \in \C^n $ such that 
\begin{align}
     \label{condicionw5} |\mathrm{Re}\, w - y_0 + \langle\omega_0 ,y_0\rangle\omega_0|^2 &< 3 \alpha^2 / 16 \\
 \intertext{and }
     \label{condicionw6} |\mathrm{Im}\, w|^2 &< 4 \alpha^2 / (4 - \beta^2). 
 \end{align} 
 Then we have
\begin{multline}\label{es:SBfromRadon}
     |\SB q ((z + \langle\omega_0, y_0\rangle) \omega_0 + w)| 
\leq C M_q (1 + \rho+|y_0|)^n\, h^{-n/2} \\ \times\e^{\frac{1}{2h}(|\mathrm{Im}\, z|^2 - |\mathrm{Re}\, z|^2)+\frac{1}{2h}|\mathrm{Im}\, w|^2} 
\Big(\e^{\frac{1}{2h}\frac{\alpha^2}{8}}  \mathcal I + \e^{-\frac{1}{2h}\frac{\alpha^2}{8}} \Big),
\end{multline}
and conditions (\ref{condicionw1})-(\ref{condicionw4}) reduce to  
\begin{equation}\label{condicionz}
|\mathrm{Re}\, z|^2 < \frac{\alpha^2}{16}, \qquad \frac{4 \alpha^2}{\beta^2} \leq |\mathrm{Im}\, z|^2 \leq \rho^2 - \frac{4 \alpha^2}{4 - \beta^2}, 
\end{equation}
with $ \rho $ large enough. 

Whenever 
$  \mathcal I \leq \e^{-\frac{\alpha^2}{8}} $,
one can choose $ h=\frac{\alpha^2 }{ 8|\log  \mathcal I |} \in (0, 1] $ so that
\[ \e^{\frac{1}{2h}\frac{\alpha^2}{8}} \mathcal I = \e^{-\frac{1}{2h}\frac{\alpha^2}{8}}. \]
From (\ref{condicionsoporte}) we have
\begin{equation}\label{es:SBgood}
|\SB q ((z + \langle\omega_0, y_0\rangle) \omega_0 + w)| \leq C M_q h^{\frac{n}{2}} \e^{\frac{1}{2h}|\mathrm{Im}\, z|^2+\frac{1}{2h}|\mathrm{Im}\, w|^2-\frac{1}{2h}(\mathrm{Re}\, z_+)^2},
\end{equation}
for all $ w \in \C^n $ and $ z \in \C $.  Recall that $ M_q = \max (1, \norm{q}{}{L^\infty(\R^n)} + \norm{q}{}{X}) $.

Consider the sub-harmonic function $ \Phi $ defined as
\begin{gather*}
\Phi (z) = |\mathrm{Re}\, z|^2 - |\mathrm{Im}\, z|^2 + 2h \log |\SB q ((z + \langle\omega_0 , y_0\rangle) \omega_0 + w)|\\
+ 2h \log \left( \frac{\e^{-\frac{1}{2h}|\mathrm{Im}\, w|^2}}{Ch^{-n/2} M_q (1 + \rho+|y_0|)^n} \right),
\end{gather*}
where the variable $ w $ has been frozen. From (\ref{es:SBgood}) and (\ref{es:SBfromRadon}) one derives that
\[ \Phi (z) < (  \mathrm{Re}\, z_{-})^2, \]
for all $ z \in \C $; and
\[ \Phi (z) < h \log \mathcal I  = - \frac{\alpha^2}{8}, \]
for all $ z \in \C $ satisfying (\ref{condicionz}).

 It is clear that $ \Phi $  satisfies  the conditions of Lemma \ref{lem:sub-Harmonic}, with parameters $a=\frac \alpha 4$, $b=\frac{2\alpha}{\beta}$ and $\lambda =\frac{\alpha^2}{8}$, hence we might conclude that
\[ \Phi(z) < - \frac{h \log(\int_{-\alpha}^\alpha (1 + |s |)^n \norm{\Rad_{y_0}q(s, \cdot)}{}{L^1(\Gamma)} \, \dd s)^{-1}}{2 \left( \cosh  ( 8 \pi / \beta ) \right)^2}, \]
for $ z \in \C $ such that
\[ |\mathrm{Re}\, z| < \frac{\alpha}{8 \cosh  ( 8 \pi / \beta )}, \qquad |\mathrm{Im}\, z| < \frac{2 \alpha}{\beta}. \]
Choosing $ \rho = 4 \alpha / \beta + |y_0| $, one can translate this estimate into the statement of the proposition, notice that (\ref{condicionw5}), (\ref{condicionw6}) and (\ref{condicionz}) follow  from (\ref{hipotesisz})  since 
$$| \mathrm{Re}\, \zeta -y_0| ^2 = |\mathrm{Re}\, z| ^2 + |\mathrm{Re}\, w - y_0 + \langle\omega_0 ,y_0\rangle\omega_0|^2$$
and $$|\mathrm{Im}\, \zeta|^2 =|\mathrm{Im}\, z|^2 +|\mathrm{Im}\, w|^2.$$
 This completes the proof of the proposition. 
\end{proof}

 \begin{proof}[Proof of Theorem \ref{th:localLOGstability}]
The key point   is the fact that the Segal-Bargmann transform restricted to real values is a convolution with the Gaussian. We exploit this by means of the following lemma. Since actually  we need a backward estimate for the heat equation, which is  an ill posed problem, we   requires the uniform  Besov  control of the potentials.

\begin{lem}\label{le:Lp_gaussianCONV}
Consider $ q \in L^p(\R^n) $ and $ G $ an open set in $ \R^n $.
Assume that there exists $ \lambda \in (0, 1) $ such that
\[ L_q := \left( \int_{\R^n} \frac{\norm{q - q(\cdot - y)}{p}{L^p(\R^n)}}{|y|^{n + \lambda p}} \, \dd y \right)^{1/p} < + \infty. \]
Then, there exists a positive constant $ C $, only depending on $ n $, such that
\[ \norm{q}{}{L^p(G)} \leq C \left( h^{-\frac{n}{2}} \norm{\SB q  }{}{L^p(G)} + L_q h^\frac{\lambda}{2} \right), \]
for all $ h \in (0, 1] $.  
\end{lem}

\begin{proof} Since
\[ q(x) = \frac{1}{(2\pi h)^\frac{n}{2}} \SB q(x) + \frac{1}{(2\pi)^\frac{n}{2}} \int_{\R^n} \e^{-\frac{1}{2}|y|^2} (q(x) - q(x - h^\frac{1}{2} y)) \, \dd y \]
almost everywhere in $ G $,
\begin{align*}
\norm{q}{}{L^p(G)} &\leq \frac{1}{(2 \pi h)^\frac{n}{2}} \norm{\SB q |_{\R^n}}{}{L^p(G)} \\ &\quad + \norm{\frac{1}{(2\pi)^\frac{n}{2}} \int_{\R^n} \e^{-\frac{1}{2}|y|^2} |q - q(\cdot - h^\frac{1}{2} y)| \, \dd y}{}{L^p(G)}.
\end{align*}
Minkowski's inequality ensures that there exists a positive constant $ C $, only depending on $ n $, such that
\begin{multline*}
\norm{\int_{\R^n} |q - q(\cdot - h^\frac{1}{2} y)| \frac{\e^{-\frac{1}{2}|y|^2}}{(2\pi)^\frac{n}{2}} \, \dd y}{p}{L^p(G)} \\
\leq C^p h^\frac{\lambda p}{2} \int_{\R^n} \frac{\norm{q - q(\cdot - y)}{p}{L^p(\R^n)}}{|y|^{n + \lambda p}} \, \dd y.
\end{multline*}
This completes the proof of the Lemma.
\end{proof}

Now we return to the proof of the Theorem.
In the case
\[ \int_{(-\alpha, \alpha)} (1 + |s|)^n \norm{\Rad_{y_0}(q)(s, \cdot)}{}{L^1(\Gamma)} \, \dd s \leq \e^{-\frac{\alpha^2}{8}} \]
holds, then Theorem \ref{th:localLOGstability} is a consequence of Proposition \ref{prop:SBnear(y_0,0)} and Lemma \ref{le:Lp_gaussianCONV} for the function $ \mathbf{1}_E q $. On the other hand, if
\[ \int_{(-\alpha, \alpha)} (1 + |s|)^n \norm{\Rad_{y_0}(q)(s, \cdot)}{}{L^1(\Gamma)} \, \dd s \geq \e^{-\frac{\alpha^2}{8}}, \]
the conclusion of the statement of Theorem \ref{th:localLOGstability} is obvious.
\end{proof}

\end{section}
\begin{section}{First application: illuminating $\Omega$ from infinity (BU)}\label{sec:BU}

We  use the solutions of the Schr\"odinger equation in the maximal domain of the  Laplace operator in $\Omega$. These solutions,  satisfying the support  condition,  were constructed by Nachman and Street in \cite{NS} in the context of \cite {KSU}, but  the construction in the case of illumination from infinity, \cite{BU},  is  easier  and follows the  same  steps. We  will collect some  estimates  for  these  solutions  which  are (some of them  implicitely) contained in \cite{BU} and  \cite{NS}.

Let  $ H(\Omega; \Delta) $ denote the elements of $ L^2(\Omega) $ such that their weak Laplacean also belong to $ L^2(\Omega) $.

\begin{lem}[Bukhgeim and Uhlmann\cite{BU}]\label{BU1} Assume that $\partial \Omega \in \mathcal{C}^2$. Then  the trace maps 
$$\tr_0 u= u |_{\partial \Omega}$$
and
$$\tr_1 u=  \partial_\nu u|_{\partial \Omega} , $$
defined in $\mathcal{C}^\infty$ have an extension, again denoted as $\tr_j$,  $j=0,1$ which is continuous from  $ H(\Omega; \Delta)$ to the Sobolev space $H^{-j-1/2}(\partial \Omega).$

If we assume in  addition that  $ \tr_0 u\in H^{3/2}$, then  $u\in H^2(\Omega)$ and 
$$\|u\|_{H^2(\Omega)}+ \|\tr_1 u\|_{H^{1/2}(\partial \Omega)}\leq C(\|u\|_{H(\Omega; \Delta)} +\|\tr_0 u\|_{H^{3/2}(\partial \Omega)}),$$
for some constant $C>0$.
 \end{lem}
The proof can  be found in \cite{BU}. Let us  remark that the definition of the extended trace maps is based on Green's formulae 
for  smooth functions. Consider $u\in H(\Omega; \Delta)$, then on one hand,
for $\omega \in H^{1/2}(\partial \Omega)$, we have
\begin{equation}\label{tr0}
\tr_0 u(\omega)= \int_{  \Omega}(u\Delta \bar v- \Delta u \bar v) \, \dd x
\end {equation}
where  $v\in H^2(\Omega)$ is  the extension
\begin{equation}\label{test1}
v|_{\partial \Omega}=0  \, ,\partial_\nu v|_{\partial \Omega}=\omega.
\end{equation}
On the other hand, for $\omega \in H^{3/2}(\partial \Omega)$, we have
\begin{equation}
\label{tr1}
\tr_1 u(\omega)= \int_{  \Omega}(u\Delta \bar v- \Delta u \bar v) \, \dd x
\end {equation}
where  $v\in H^2(\Omega)$ is the extension
\begin{equation}\label{test2}
v|_{\partial \Omega}=\omega  \, ,\partial_\nu v|_{\partial \Omega}=0.
\end{equation}
The generalized  Green's formula reads as follows.
\begin{lem}
For $u \in H(\Omega; \Delta)$ and $v\in H^2(\Omega)$, we have
\begin{equation}
\int_\Omega(\Delta-q)u\bar v \, \dd x=\int_\Omega u \overline{(\Delta-\bar q)v}  \, \dd x + \duality{\tr_1u}{\overline{\tr_0v} }- \duality{\tr_0u}{\overline{\tr_1v}}.
\end {equation}
\end{lem}

\subsection{The Dirichlet-to-Neuman map}

The next step is to define the Dirichlet-to-Neumann map  associated  to the Schr\"odinger equation $-\Delta + q$. To achieve  a definition in a extended domain that contains the traces of the solutions $u \in H(\Omega; \Delta)$ of the   equation 
$(-\Delta + q)u=0$, we will need the following lemma (see \cite{NS}).

We will denote $\mathcal{H}(\partial \Omega)$ the range  of the map $$\tr_0:   H(\Omega; \Delta)\to H^{-1/2}(\partial \Omega),$$
and also consider  the space of solutions of Schr\"odinger equation 
$$b_q:=\{ u\in L^2(\Omega): (-\Delta + q)u=0\} \subset   H(\Omega; \Delta).$$
Then  we have 
\begin{lem} [Nachman and Street \cite{NS}]\label{nast3}If $q \in L^{\infty}(\Omega)$ and $0$ is not  a Dirichlet  eigenvalue  of $-\Delta + q$ in $\Omega$ then the trace  map
$$\tr_0: b_q \to  \mathcal{H}(\partial \Omega)$$ is one to one  and onto.
\end{lem}
We consider the inverse maps 
$$P_q= \tr_0^{-1}:   \mathcal{H}(\partial \Omega)\to b_q$$
and  define the norm  of  $\mathcal{H}(\partial \Omega)$  as 
$$\|\phi\|_{\mathcal{H}(\partial \Omega)}:= \|P_0\phi\|_{L^2(\Omega)}$$
Then  we have:
\begin{lem}[Nachman and Street \cite{NS}]\label{nast1} The map  $\tr_0: H(\Omega; \Delta)\to \mathcal{H}(\partial \Omega)$ is continuous and,
under the hypothesis  of the  previous lemma, the map  $$\tr_0: b_q\to \mathcal{H}(\partial \Omega)$$ is a homeomorphism.
\end{lem}
One can define   the Dirichlet-to-Neumann map 
\begin{equation}\Lambda_q: \mathcal{H}(\partial \Omega)\to H^{-3/2}(\partial \Omega)
\end{equation} 
as the map     $$\Lambda_q(\phi)  =  \tr_1(P_q(\phi)).$$
To be more precise, by (\ref{tr1}) for $\phi\in\mathcal{H}(\partial \Omega)$ and $\psi \in H^{3/2}(\partial \Omega)$,  we have
\begin{equation}\label{Dirichlet-to-Neumann}
\duality{\Lambda_q\phi}{\psi}=\int_\Omega(P_q\phi\Delta \bar v-qP_q\phi\bar v) \, \dd x,
\end{equation}
where $v$ is the extension in (\ref{test2}). It would be desirable to construct  the Dirichlet-to-Neumann  map  as  a selfdual  operator,  unfortunately this can not be  achieved, instead we have:

\begin{lem} [Nachman and Street \cite{NS}]\label{nast2}
Let $q_j$ , $j=1,2$ be $L^{\infty}$ potentials so that  $0$ is not  a Dirichlet  eigenvalue  of $-\Delta + q_j$ in $\Omega$.  Then 
$\Lambda_{q_2}- \Lambda_{q_1}$  extends to a continuous  map
$\mathcal{H}(\partial \Omega)\to\mathcal{H}(\partial \Omega)^*$.
\end{lem}
These lemmas  can be found in \cite{NS}.
Let us remark that we have, if $\phi, \psi \in \mathcal{H}(\partial \Omega)$,
\begin{equation}\label{basic1}
 \duality{(\Lambda_{q_2}- \Lambda_{q_1}) (\phi)}{ \psi}= \int_\Omega P_{q_1}(\phi)(q_1-q_2)P_{q_2} (\psi) \, \dd x.
\end{equation}
This  formula is  the starting point of the recovery of values of the error in the interior of the domain.
Notice that
$$\|\Lambda_{q_1}- \Lambda_{q_2}\|_{\mathcal{H}(\partial \Omega)\to\mathcal{H}(\partial \Omega)^*}\geq \|\Lambda_{q_1}- \Lambda_{q_2}\|_{H^{1/2}(\partial \Omega) \to H^{-1/2}(\partial \Omega)}. $$

\subsection{The partial data map.} Given  $N\subset \Sph^{n-1}$  open (which could be very small) and $\xi \in N$, We assume $F\subset\d\Omega$  to be  a neighborhood of $\partial  \Omega_-(\xi) $
for any $\xi \in N $ and $B$ a neighborhood of $\partial  \Omega_+(\xi)$ for any $\xi \in N $.

Given two potentials satisfying the hypothesis of Lemma \ref{nast2},  we will consider   the difference of their partial data  measurements as
\begin{equation}\label{Dirichlet-to-Neumann2}
 \|\Lambda_{q_1}-\Lambda_{q_2}\|^*_{B\to F}:=\\  
\sup\Big\{\Big|\duality{(\Lambda_{q_1}- \Lambda_{q_2})\phi_B}{\phi_F}\Big|\Big\},
\end{equation}
where $\sup $ is taken over  the  set  
\[\big\{(\phi_B,\psi_F):\|\phi_B\|_{ \mathcal{H}(\partial \Omega)}= \|\psi_B\|_{ \mathcal{H}(\partial \Omega)}=1, \phi_B \in  \mathcal {E }'(B) \text{ and }
\psi_F\in   \mathcal {E }'(F)\big\}.\]

\subsection{Fadeev's special solutions}
We collect   and remark  results in \cite{BU} and \cite{NS}, concerning the existence and  a priori bounds of $H( \Omega,\Delta)$ solutions  of the Sch\"rodinger equation adapted  to the support requirements of the partial data.

The modifications of \cite{NS} were done in the  context of partial data considered in \cite{KSU}, but the construction can be adapted to the case of \cite{BU}, the only point is to write the operator conjugated with the  exponential  weight  in the appropriate coordinates so that it is given as perturbations of the laplacean with terms of the complex operators $\d_z$ and $\d_{\bar z}$. The final output of this construction is as  follows.

 We consider $q$ is as in Lemma \ref{nast3} , $\xi$  and $\zeta$ unit orthogonal vectors so that $\xi\in N$. 
  We write $x=(x_1,x_2, x'')$ with respect to  an orthonormal  basis $\{e_1,...,e_n\} $ so that $e_1=\xi$ and $e_2=\zeta$ and $x''\in \R^{n-2}$.

\begin{thm} \label{NSsolution}

For $\tau \geq 1$ sufficiently large and $g\in C^\infty(\R^{n-2})$, and $B$ and $F$ as in the  statement of Theorem \ref{BUstability},
there exists a unique solution $w_\tau \in H(\Omega; \Delta)$ of the  equation $(-\Delta + q)w_\tau =0$ in $\Omega$, such that $\tr_0w_\tau  \in\mathcal{H}(\partial \Omega)\cap \mathcal {E }'(B) $ and which can be written as
 \[w_\tau(x)=\e^{\tau\langle \xi+ i\zeta, x \rangle}(g(x'')+ R(\tau, x))\]
 where \[\|R(\tau,\cdot)\|_{L^2(\Omega)}\leq C \frac1 \tau(\| qg\|_{L^2(\Omega)} +  \tau^{1/2}\|\Delta g\|_{L^2(\Omega)}).
  \]
The same is true changing $\tau$ by $-\tau$  and $B$ by $F$.

\end{thm}
The key ingredients in the proof are   boundary Carleman estimates and  some orthogonality properties  of the  reduced  data solutions.

\begin{prop} Let   $q\in L^\infty(\Omega)$, there exist  $\tau_0>0$ and $C>0$  such that for all  $u\in \mathcal{C}^\infty(\overline{\Omega})$  , $u|_{\partial\Omega}=0$  and $\tau>\tau_0$
\begin{multline} C\tau^2\int_\Omega |\e^{-\tau \langle x,\xi\rangle}u|^2 \, \dd x+ \tau \int_{\partial \Omega_+}\langle x,\xi\rangle|e^{-\tau x\langle x,\xi\rangle}\partial_\nu u|^2 \, \dd A \\
\leq
\int_\Omega |\e^{-\tau \langle x,\xi\rangle}(\Delta -q)u |^2 \, \dd x- \tau \int_{\partial \Omega_-}\langle x,\xi\rangle |\e^{-\tau x\cdot \xi}\partial_\nu u|^2 \, \dd A. 
\end{multline}
\end{prop}

\subsection{\bf Stability from the Dirichlet-to-Neumann map to the Radon transform.}
We will use identity  (\ref{basic1})  together with the special solutions of Theorem  \ref{NSsolution}  to prove for $q=(q_1-q_2)\mathbf{1}_\Omega$:

\begin{prop} 
\label{PropDualEst}
For any $g\in C^\infty(\R^{n-2})$ there exists  $C>0 $ which   only depends on   $\Omega$  and  the  a priori bound of $ \|q_j\|_{L^\infty}$, such that
\begin{gather*}
\sup_{\xi\in N, \zeta \in \xi^{\perp}}\bigg|\int_{ [\xi, \zeta]^\perp }g(x'')\int_{\R^2}q(x'' + t\xi+s\zeta) \,\dd t \, \dd s \, \dd x'' \bigg|\\
\leq C(\tau^{-1/2}+\e^{c\tau} \|\Lambda_{q_1}-\Lambda_{q_2}\|^*_{B\to F} )\|g\|_{H^2(\Omega)},
\end{gather*}
where $ [\xi, \zeta] $ denotes the plane  spanned by $\xi$ and $\zeta$.  
\end{prop}
\begin{clry} 
If we consider the open set in $\Sph^{n-1}$
$$M= \cup_{\xi \in N} [\xi]^{\perp},$$
then we have in natural coordinates of the Radon transform
\begin{multline}\label{controlradon}
\sup_{\eta\in M}\bigg|\int_{\R}\tilde{g}(r)\Rad q(  r, \eta ) \, \dd r \bigg|
\\ \leq C(\tau^{-1/2} + \e^{c\tau} \|\Lambda_{q_1}-\Lambda_{q_2}\|_{B\to F}^*)\|\tilde g\|_{H^2(\R)}.
\end{multline}
\end{clry}

\begin{proof} With the notation of the Proposition, fix  $\eta\in [\xi, \zeta]^\perp $ with $|\eta|=1$. We  write $x''\in [\xi, \zeta]^\perp$ as $x''= x'''+r\eta$ where $x'''\in [\eta, \xi, \zeta]^\perp$ and choose appropriate $g$ independent of $x'''$, $\tilde g(r)=g(r\eta)$ in the Proposition to  get  the control of the Radon transform on any hyperplane, obtained by translation  of  hyperplanes in the pencil   through the origin that  contains $\xi$. We write $x= r\eta+ x'''+ t\xi+ s\zeta$  with $x'''\in [\eta, \xi, \zeta]^\perp$, so that $\dd x''= \dd x'''\dd r$.
Hence
\begin{gather*}
\sup_{\xi\in N, \zeta \in \xi^{\perp}}\bigg|\int_{\R}g(r\eta)\int_{\R^{n-1}}q(x'''+ r\eta + t\xi+s\zeta) \, \dd t \, \dd s \, \dd x'''  \, \dd r \bigg| \\
\leq C(\tau^{-1/2} + \e^{c\tau} \|\Lambda_{q_1}-\Lambda_{q_2}\|_{B\to F}^*)\|\tilde g\|_{H^2(\R)}
\end{gather*}
Translating the above  in the natural variables $(r, \eta)$ of the Radon transform  gives the Corollary.
\end{proof}

\begin{proof}[Proof of  Proposition \ref{PropDualEst}]

By Lemma \ref{NSsolution} there exist  $v_1\in  b_{q_1}$ such that 
$\supp \tr_0v_1\subset B$ of the form
$$v_1= \e^{-\tau\langle\xi+i\zeta, x\rangle}( 1+R_1( \tau, x))$$ and  $v_2\in b_{q_2}$ such that $\supp \tr _0 v_2\subset F$ and
$$v_2= \e^{\tau\langle\xi-i\zeta, x\rangle}( g(x'')+R_2( \tau, x)).$$
We might write (\ref{basic1}), by using these  solutions as  
\begin{equation*} 
 \duality{(\Lambda_{q_2}- \Lambda_{q_1}) (\tr_0 v_1)}{ \tr _0v_2}= \int_\Omega v_1 (q_1-q_2)\bar v_2 \, \dd x.
\end{equation*}
We also have
\begin{equation*}
 \left| \duality{(\Lambda_{q_2}- \Lambda_{q_1}) (\tr_0 v_1)}{ \tr _0v_2} \right| \leq
  \|\Lambda_{q_1}-\Lambda_{q_2}\|^*_{B\to F}\|\tr_0 v_1\|_{\mathcal{H}(\partial \Omega)}\|\tr_0 v_2\|_{\mathcal{H}(\partial \Omega)},
   \end{equation*}
    which from Lemma \ref{nast1} gives
   $$\left| \duality{(\Lambda_{q_2}- \Lambda_{q_1}) (\tr_0 v_1)}{ \tr _0v_2} \right|
   \\ \leq C \|\Lambda_{q_1}-\Lambda_{q_2}\|^*_{B\to F}\|v_1\|_{H(\Omega; \Delta)}\|\bar v_2\|_{H(\Omega; \Delta)}$$
Since 
  \begin{align*}
\|v_2\|_{H(\Omega; \Delta)}&\leq C(\|q_2\|_{L^\infty}+ 1)\|v_2\|_{L^2(\Omega)}\\
&\leq C(\|q_2\|_{L^\infty}+ 1) \sup_{x\in \Omega} \e^{\tau\langle \xi+\zeta, x\rangle}\left(\|g\|_{L^2(\Omega)}+\|R_2(\tau, \cdot)\|_{L^2(\Omega)} \right)\\
&\leq C(\|q_2\|_{L^\infty}+ 1)  \e^{C\tau } \|g\|_{H^2(\Omega)}. 
\end{align*}
In a similar way, 
\[
\|v_1\|_{H(\Omega; \Delta)}\leq C(\|q_1\|_{L^\infty}+ 1)  \e^{C\tau }. 
\]
Hence
  \begin{gather*} 
\left| \duality{(\Lambda_{q_2}- \Lambda_{q_1}) (\tr_0 v_1)}{ \tr _0v_2} \right| \leq
 C \e^{C\tau } \|\Lambda_{q_1}-\Lambda_{q_2}\|^*_{B\to F}    \|g\|_{H^2(\Omega)} ,
 \end{gather*}
 with constants only depending on $\Omega$  and  the  a priori bound assumed on $ \|q_j\|_{L^\infty}$.
 
 We have
 \[\int_\Omega q v_1  \bar v_2 \, \dd x=\int_\Omega q( 1+R_1( \tau, x))(\overline{ g(x'')+R_2( \tau, x))} \, \dd x, 
 \]
 hence
 \[\bigg|\int_\Omega v_1 q\bar v_2 \, \dd x \bigg| \geq \bigg|\int_\Omega   q\bar g(x'') \, \dd x \bigg|
 - C\tau^{-1/2}\|g\|_{H^2(\Omega)}\]
 where $C$ only depends on   $\Omega$  and  the  \textit{a priori} bound of $ \|q_j\|_{L^\infty}$.
 
Putting all together, we get
  \begin{gather*} 
  \bigg|\int_\Omega   q\bar g(x'') \, \dd x\bigg|\leq 
  C \e^{C\tau } \|\Lambda_{q_1}-\Lambda_{q_2}\|^*_{B\to F}    \|g\|_{H^2(\Omega)} +C\tau^{-1/2}\|g\|_{H^2(\Omega)}.
  \end{gather*}
This ends the proof of Proposition \ref{PropDualEst}.
\end{proof}

\begin{thm}[Stability of Radon transform]
We have the following estimate on $q=(q_1-q_2)\mathbf{1}_{\Omega}$
\begin{multline}
\label{radon}
\left(\int_M\left(\int_{\R}|\Rad q(r,\eta)|^2 \, \dd r \right)^{\frac{n+3}4} \, \dd\sigma(\eta)\right)^{\frac 2{n+3}}\\
\leq C(\tau^{-1/2}+\e^{c\tau} \|\Lambda_{q_1}-\Lambda_{q_2}\|^\ast_{B\to F})^ \frac{n-1}{n+3}.
\end{multline} 
\end{thm}

\begin{proof} The estimate can be obtained by interpolation of (\ref{controlradon}) and the following estimate  for the Radon transform 
$$\int_{{\Sph}^{n-1 }}\int_{\R}(1+\tau^2)^{\frac{n-1}{2}}|\widehat{\Rad q}(\tau,\eta) |^2  \, \dd \tau \, \dd\sigma(\eta) \leq {  \|q\|_{L ^2}^2 }.$$
This estimate can be found in \cite{Nat}.
\end{proof}

\subsection{End of proof of Theorem \ref{BUstability}}

 \begin{proof}
 
 Our aim is to use  Theorem  \ref{th:localLOGstability} together  with  estimate \eqref{radon}. We need  the supporting  condition (b)  in Theorem  \ref{th:localLOGstability}.  To achieve  this condition  
let us take  $\eta \in M$, we know by translation that there exists $y_0 \in \supp q$ such that the hyperplane $ H_{y_0}(0,\eta)$, see \eqref{hiperp}, which  contains  $y_0$ satisfies  condition (b).

Now  $M$ is  a neighbourhood of $\eta$ in the sphere and from the previous  Theorem we can control the Radon transform  for $\omega \in M$  and $s \in \mathbb R$. 

We can take $\alpha>0$ large enough, so that  $G $  in (\ref{abierto}) contains $\supp q$ (the $\beta $ depends on the size of $M$). Then
\begin{multline*}\int_{(-\alpha, \alpha)} (1 + |s-\langle \omega,\zeta-y_0 \rangle|)^n \norm{\Rad_{y_0}(q_1 - q_2)(s, \cdot)}{}{L^1(M)} \, \dd s \\
\leq C(\supp q, \beta)
\left(\int_M\left(\int_{\R}|\Rad q(r,\eta)|^2 \, \dd r \right)^{\frac{n+3}4} \, \dd\sigma(\eta)\right)^{\frac 2{n+3}}
\end{multline*}
The choice $\tau=\frac{1}{2c}|\log \|\Lambda_{q_1}-\Lambda_{q_2}\|^*_{B\to F}|$ in the estimate \eqref{radon} gives
\begin{multline*}\int_{(-\alpha, \alpha)} (1 + |s-\langle \omega,\zeta-y_0 \rangle|)^n \norm{\Rad_{y_0}(q_1 - q_2)(s, \cdot)}{}{L^1(M)} \, \dd s 
\leq C(\supp q, \beta) \\ \times
\left((C|\log \|\Lambda_{q_1}-\Lambda_{q_2}\|^*_{B\to F}|^{-1/2}+ ( \|\Lambda_{q_1}-\Lambda_{q_2}\|^*_{B\to F} )^{1/2}\right)^  \frac {n-1}{n+3}.
\end{multline*}  
This  together  with (\ref{es:LOGstabiltyRADON}) gives the desired estimate
\begin{gather*}
\norm{q_1 - q_2}{}{L^p(\Omega)} \leq C \big|\log|\log\|\Lambda_{q_1}-\Lambda_{q_2}\|^*_{B\to F}| \big|^{-\frac{\lambda}{2}},
\end{gather*}
 for small norms  of the difference of Dirichlet-to-Neumann maps.
 
This ends the proof of Theorem \ref{BUstability}.
\end{proof}

\end{section}

\begin{section}{Second Application: illuminating $\Omega$ from  a point (KSU)}\label{sec:KSU}

To obtain stability in the case of \cite{KSU}, we could have  followed  the  same  approach, but in order to  prove  stability  with the usual $H^{1/2}(\d\Omega)\to H^{-1/2}(\d\Omega)$-norm of the Dirichlet-to-Neumann map we  will recall on the   $ H^1 $-solutions of  the Schr\"odinger  equation constructed  by Chung (see \cite{Ch}), in his  work on  the  magnetic case  with  partial  data.

As above, we consider two potentials $ q_1 $ and $ q_2 $ being in $ L^\infty (\Omega) $ and such that $ 0 $ is not an eigenvalue of $ (- \Delta + q_j) : H^1_0(\Omega) \cap H(\Omega; \Delta) \longrightarrow L^2(\Omega) $ for $ j \in \{1, 2\} $. Let $ \Lambda_{q_j} $ denote the Dirichlet-to-Neumann map corresponding to the coefficient $ q_j $. 

Let $ \langle \cdot | \cdot \rangle $ denote the duality between $ H^{1/2} (\partial \Omega) $ and $ H^{-1/2} (\partial \Omega) $ and recall that $ \Lambda_{q_j} : H^{1/2} (\partial \Omega) \longrightarrow H^{-1/2} (\partial \Omega) $ is defined as
\[ \duality{\Lambda_{q_j} f_j}{g} := \int_\Omega\langle \nabla u_j , \nabla v\rangle  \, \dd x + \int_\Omega q_j u_j v \, \dd x \]
for any $ f_j, g \in H^{1/2} (\partial \Omega) $, where $ u_j \in H^1(\Omega) $ is the weak solution of the Schr\"odinger equation $ (-\Delta + q_j) u_j = 0 $ in $ \Omega $ with $ u_j|_{\partial \Omega} = f_j $ and $ v \in H^1(\Omega) $ with $ v|_{\partial \Omega} = g $. Since $ 0 $ is not a Dirichlet eigenvalue, $ \Lambda_{q_j} $ is a well-defined  bounded linear  operator.

An integration by parts shows that, for any $ v_j \in H^1(\Omega) $ weak solution of the Schr\"odinger equation
\begin{equation}
(-\Delta + q_j) v_j = 0 \label{eq:SCHschEQ}
\end{equation}
in $ \Omega $, one has
\begin{equation}
\duality{(\Lambda_{q_1} - \Lambda_{q_2})(v_1|_{\partial \Omega})}{v_2|_{\partial \Omega}} = \int_{\Omega} (q_1 - q_2) v_1 v_2 \, \dd x \label{eq:SCHin-out}. \\
\end{equation} 
Let $ y_0 $ be a point out of $ \mathrm{ch}(\Omega) $, the convex hull of $ \Omega $, and consider $ N $ a neighbourhood of $ y_0 $ also away from $ \mathrm{ch}(\Omega) $. Consider a hyperplane $ H $ separating $ N $ and $ \mathrm{ch}(\Omega) $ and let $ H^+ $ denote the semi-space delimited by $ H $ and containing $ \overline{\Omega} $. 
Consider $ R $ a positive constant such that, for all $ y \in N $, $ \Omega \subset B(y,R) $ and set $ \Sigma $ a subset of $ \Sph^{n - 1} $ with non-empty interior and 
$$\bar \Sigma\subset \{ \theta \in \Sph^{n - 1} \text{ such that for all } y \in N ,  \, y + R\theta \notin H^+   \text{ and }  y - R\theta \notin H^+ \}.$$  

Let $ F $ and $ B $ denote  open neighborhoods, for any $ y \in N $, of the faces $\d\Omega_-(y) $ and $\d\Omega_+(y) $ respectively. Consider $ \chi \in C^\infty (\partial \Omega) $ satisfying $ \chi : \partial \Omega \longrightarrow [0, 1] $, $ \supp \chi \subset F $ and $ \chi |_{F_\varepsilon} = 1 $ with $ \partial \Omega_-(y) \subset F_\varepsilon \subseteq \overline{F_\varepsilon} \subset F $ ($ \chi $ and $ F_\varepsilon $ may depend on $ y $). Obviously, we can write
\begin{multline*}
\duality{(\Lambda_{q_1} - \Lambda_{q_2})(v_1|_{\partial \Omega})}{v_2|_{\partial \Omega}} = \duality{(\Lambda_{q_1} - \Lambda_{q_2})(v_1|_{\partial \Omega})}{(1 - \chi) v_2|_{\partial \Omega}} \\
+ \duality{(\Lambda_{q_1} - \Lambda_{q_2})(v_1|_{\partial \Omega})}{\chi v_2|_{\partial \Omega}}.
\end{multline*}
Let $ w_2 \in H^1(\Omega) $ denote the weak solution of $ (- \Delta + q_2) w_2 = 0 $ in $ \Omega $ with $ w_2|_{\partial \Omega} = v_1|_{\partial \Omega} $. This implies $ \partial_\nu( v_1 - w_2 )|_{\partial \Omega} \in H^{1/2} (\partial \Omega) $ ($ \nu $ stands for the unit exterior normal vector on $ \partial \Omega $), since $ \partial \Omega $ is smooth enough (see \cite{BU}). Therefore,
\begin{multline*}
\left| \duality{(\Lambda_{q_1} - \Lambda_{q_2})(v_1|_{\partial \Omega})}{v_2|_{\partial \Omega}} \right| \leq \left| \int_{\partial \Omega} \partial_\nu( v_1 - w_2 ) (1 - \chi) v_2 \, \dd A \right| \\
+ \norm{(\Lambda_{q_1} - \Lambda_{q_2})(v_1|_{\partial \Omega})}{}{H^{-1/2}(F)} \norm{\chi v_2}{}{H^{1/2}_0(F)}
\end{multline*}
($ H^{1/2}_0(F) $ is the subspace of $ H^{1/2}(\partial \Omega)$ whose elements have their support in $ \overline{F} $, $ H^{-1/2}(F) $ is nothing but its dual --more about these spaces can be found in \cite{Ca11}).

Assuming $ \supp \, v_1|_{\partial \Omega} \subset \overline{B} $, one has
\begin{multline*}
\left| \duality{(\Lambda_{q_1} - \Lambda_{q_2})(v_1|_{\partial \Omega})}{v_2|_{\partial \Omega}} \right| \leq \int_{\partial \Omega \setminus F_\varepsilon} |\partial_\nu( v_1 - w_2 )| |v_2| \, \dd A \\
+ \norm{\Lambda_{q_1} - \Lambda_{q_2}}{}{ } \norm{v_1}{}{H^{1/2}_0(B)} \norm{\chi}{}{C^{0, 1/2 + \epsilon}(\partial \Omega)} \norm{v_2}{}{H^{1/2}(\partial \Omega)},
\end{multline*}
where $  \norm{\cdot}{}{ }= \norm{\cdot}{}{B\to F} $ denotes the operator norm from $ H^{1/2}_0 (B) $ to $ H^{-1/2} (F) $ and $ \epsilon $ is any positive constant ($ C^{0, 1/2 + \epsilon}(\partial \Omega) $ is the space of $ (1/2 + \epsilon) $-H\"older functions defined on the $ \partial \Omega $ --a close subset of $ \R^n $). 

For future references,
\begin{multline} 
\left| \duality{(\Lambda_{q_1} - \Lambda_{q_2})(v_1|_{\partial \Omega})}{v_2|_{\partial \Omega}} \right| \leq \int_{\partial \Omega \setminus F_\varepsilon} |\partial_\nu( v_1 - w_2 )| |v_2| \, \dd A \\
+ C \norm{\Lambda_{q_1} - \Lambda_{q_2}}{}{} \norm{v_1}{}{H^{1/2}_0(B)} \norm{v_2}{}{H^{1/2}(\partial \Omega)}. \label{es:SCHfromBOUNDARYtoINTERIOR}
\end{multline}
  Since  we assume that $ 0 $ is not an eigenvalue of $ (- \Delta + q_j) : H^1_0(\Omega) \cap H(\Omega; \Delta) \longrightarrow L^2(\Omega) $we have that $ \partial_\nu( v_1 - w_2 )|_{\partial \Omega} \in L^{2} (\partial \Omega) $. 
  
\subsection{Controlling the Radon transform}
Estimate (\ref{es:SCHfromBOUNDARYtoINTERIOR}) will be the starting point to control the 2-plane transform of $ q \in L^\infty(\R^n) $, for  $ q= (q_1 - q_2) \mathbf{1}_\Omega  $.   In order to obtain this information, we need to plug in (\ref{es:SCHfromBOUNDARYtoINTERIOR}) special solutions of the Schr\"odinger equation (\ref{eq:SCHschEQ}). These special solutions are a generalization of the classical complex geometrical optic solutions (or Fadeev solutions) and were constructed by Kenig, Sj\"ostrand and Uhlmann \cite{KSU} and recently by Chung \cite{Ch}. More precisely, from \cite{KSU} (see Lemma 3.4 in  \cite{DKSU1}) we have,

\begin{prop} There exists a solution $v_2$ of $(-\Delta + q_2) v = 0 $ of the  form
\[ v_2(\tau) = \e^{\tau (\phi_2 + i \psi_2)}(a_2 + r_2(\tau)), \]
where $ \tau $ is a large parameter, $ \phi_2 : (\R^n \setminus N) \times N \longrightarrow \R $ defined by
\[ \phi_2(x, y) = - \log |x - y|, \]
$ \psi_2 : \tilde{\Omega} \times N \times \Sigma \longrightarrow \R $ defined by
\[ \psi_2(x, y, \theta) = d_{\Sph^{n - 1}} \left( \frac{x - y}{|x - y|}, \theta \right) \]
with $ \tilde{\Omega} = \cap_{y \in N} \left( H_+ \cap \{ x \in \R^n : |x - y| < R \} \right) $ and $d_{\Sph^{n - 1}} $ the geodesic distance on $ \Sph^{n - 1} $ associated to the Euclidean metric restricted to $ \Sph^{n - 1} $. Furthermore, $ a_2 : \tilde{\Omega} \times N \times \Sigma \longrightarrow \C $ defined as
\[ a_2(x, y, \theta) = \left( 2 |x - y - \theta \cdot (x - y) \theta| \right)^{- \frac{n - 2}{2}} g\left( \frac{x - y - \theta \cdot (x - y) \theta}{|x - y - \theta \cdot (x - y) \theta|} \right), \]
with any $ g : \Sph^{n - 2} \longrightarrow \C $ smooth, $ r_2(\tau) \in H^1(\Omega) $ and satisfies
\begin{multline*}
\tau \norm{r_2(\tau)}{}{L^2(\Omega)} + \tau^{1/2} \norm{r_2(\tau)}{}{L^2(\partial \Omega)} + \norm{\nabla r_2(\tau)}{}{L^2(\Omega)^n}  \\
\leq C  \left( \norm{g}{}{L^2(\Sph^{n - 2})} + \norm{\Delta_{\Sph^{n - 2}} g}{}{L^2(\Sph^{n - 2})} \right)
\end{multline*}
where $ \Delta_{\Sph^{n - 2}} $ is the Laplace-Beltrami operator on $ \Sph^{n - 2} $ for the canonical metric on $ \Sph^{n - 2} $. Here $ C $ depends on $ \norm{q_2}{}{L^\infty (\Omega)} $, on $ \Omega $, on the distance of $ N $ to the hyperplane $ H $.
\end{prop}
Let us denote 
$$Z(y)=\{x\in\d\Omega: \langle x-y,\nu(x)\rangle =0\},$$
 and assume that    $ E $ is a compact subset of $\d\Omega_-(y) \setminus Z(y) $ with $ E \supset \partial \Omega \setminus \overline{B} $. Let us recall Proposition 7.2 in \cite{Ch}.

\begin{prop}
There exists a solution $v_1$  of  $v_1$ of $(-\Delta + q_1) v = 0 $  vanishing on $ E $ and having the form
\[ v_1(\tau) = \e^{\tau (\phi_1 + i \psi_1)}(a_1 + r_1(\tau)) - \e^{\tau l} b, \]
where $ \tau $ is the same as above, $ \phi_1 : (\R^n \setminus N) \times N \longrightarrow \R $ defined by
\[ \phi_1(x, y) = -\phi_2(x, y), \]
$ \psi_1 : \tilde{\Omega} \times N \times \Sigma \longrightarrow \R $ defined by
\[ \psi_1(x, y, \theta) = - \psi_2(x, y, \theta), \]
$ a_1 : \tilde{\Omega} \times N \times \Sigma \longrightarrow \C $ defined as
\[ a_1(x, y, \theta) = \left( 2 |x - y - \theta \cdot (x - y) \theta| \right)^{- \frac{n - 2}{2}}, \]
additionally $ r_1(\tau) \in H^1(\Omega) $ and satisfies
\[ \tau \norm{r_1(\tau)}{}{L^2(\Omega)} + \norm{\nabla r_1(\tau)}{}{L^2(\Omega)^n} \leq C \]
where $ C $ depends again on $ \norm{q_1}{}{L^\infty (\Omega)} $. 
 Finally, $ l : \tilde{\Omega} \times N \times \Sigma \longrightarrow \C $ is smooth and it satisfies $ \mathrm{Re}\, l = \varphi_1 - k $ with $ k(x) \simeq \mathrm{dist}(x, E) $ in $ G $, a neighbourhood of $ E $ with non-empty interior in $ \R^n $, and $ b : \tilde{\Omega} \times N \times \Sigma \longrightarrow \C $ is twice continuously differentiable in $ \tilde{\Omega} $ and $ \mathrm{supp}\, b \subset G $.
\end{prop}

  Let us remark recall the definition of the 2-plane transform. Given $y\in \R^n$, and $\theta$ and $\eta$ unitary orthogonal vectors, we denote
\[ Rq (y, \theta, \eta) := \int_{\R \times \R} q(y + t \theta + r \eta) \, \dd t \, \dd r. \]
This is just the integral of $ q $ in the plane $\{ y \} + [ \theta, \eta ] $, where $[ \theta, \eta ]$ denotes the plane  spanned by $\theta$ and $\eta$. Notice that there is  some redundancy on variables, since it is  enough  to  define the  above  for  $y\in [ \theta, \eta ]^\perp $,  see \cite {So}.

  We assume that $N $ contains the ball $B(0,\alpha)$ and  that   $ \norm{\Lambda_{q_1} - \Lambda_{q_2}}{}{} \leq 1 .$ 
  
\begin{thm}\label{twoplanes} The following estimate holds
\begin{multline}\label{2-plane1}
\sup_{y\in B(0,\alpha) , \theta \in \Sigma} \|Rq( y,  \theta ,\eta)\|_{H^{-2}(S_\theta)}\\ \leq  C \left( \tau^{-1/4} \norm{q}{1/2}{L^\infty(\R^n)} +  \e^{c\tau} \norm{\Lambda_{q_1} - \Lambda_{q_2}}{1/4}{B\to F} \right),
\end{multline}
where $ S_\theta =  \Sph^{n - 1}\cap {\theta}^\perp  $ and we consider  the  measure $ \dd \sigma $, the volume form on $ S_\theta $ associated to the canonical metric on $ S_\theta $.
\end{thm}

\begin{proof}

We next plug   in (\ref{eq:SCHin-out}) the solutions in the above Propositions and bound by below the absolute value of the term in the right hand side of this identity:
\begin{multline*}
\left| \int_{\Omega} q \, v_1 v_2 \, \dd x \right| \geq  \left| \int_{\tilde{\Omega}} q\, a_1 a_2 \, \dd x \right| \\ - \norm{q}{}{L^\infty(\R^n)} 
 \times \Big(\norm{b}{}{L^\infty(\Omega \cap G)} \left( \norm{a_2}{}{L^2(\Omega)} + \norm{r_2}{}{L^2(\Omega)} \right) 
\norm{\e^{- \tau k}}{}{L^2(\Omega \cap G)} \\   + \norm{a_2}{}{L^2(\Omega)} \norm{r_1}{}{L^2(\Omega)}
 + \norm{a_1}{}{L^2(\Omega)} \norm{r_2}{}{L^2(\Omega)} + \norm{r_1}{}{L^2(\Omega)} \norm{r_2}{}{L^2(\Omega)}\Big).
\end{multline*}
The last inequality, identity (\ref{eq:SCHin-out}), the properties of the solutions $ v_1 $ and $ v_2 $ and (\ref{es:SCHfromBOUNDARYtoINTERIOR}) imply that
\begin{multline}
\left| \int_{\tilde{\Omega}} q\, a_1 a_2 \, \dd x \right| \leq \frac{C}{\tau^{1/2}} \norm{q}{}{L^\infty(\R^n)} \left( \norm{g}{}{L^2(\Sph^{n - 2})} + \norm{\Delta_{\Sph^{n - 2}} g}{}{L^2(\Sph^{n - 2})} \right) \\ 
\label{nuevo}
+ \int_{\partial \Omega \setminus F_\varepsilon} |\partial_\nu( v_1 - w_2 )| |v_2| \, \dd A + C 
\norm{\Lambda_{q_1} - \Lambda_{q_2}}{}{} \e^{c \tau} \norm{g}{}{H^1(\Sph^{n - 2})}.
\end{multline}
Mind that $ \tau^{-1/2} $ comes from $ \|\e^{- \tau k} \|_{L^2(\Omega \cap G)} $, dependences of $ C $ has not changed and $ c $ only depends on $ \Omega $.

We are next going to estimate the boundary integral term in \eqref{nuevo}. In order to do so, we are going to choose $ F_\varepsilon $ in such a way that $ \partial \Omega \setminus F_\varepsilon = \{ x \in \partial \Omega : (x - y) \cdot \nu(x) \geq \varepsilon \} $ with $ y \in N $. Thus,
\begin{align*}
\int_{\partial \Omega \setminus F_\varepsilon} |\partial_\nu( v_1 - w_2 )| |v_2| \, \dd A 
&= \int_{\partial \Omega \setminus F_\varepsilon} \e^{\tau \phi_2} |\partial_\nu( v_1 - w_2 )| |a_2 + r_2| \, \dd A \\
&\leq C \left( \int_{\partial \Omega \setminus F_\varepsilon} \e^{2 \tau \phi_2} |\partial_\nu( v_1 - w_2 )|^2 \, \dd A \right)^{\frac{1}{2}} 
\\ &\qquad \times\left( \norm{g}{}{H^1(\Sph^{n - 2})} + \norm{\Delta_{\Sph^{n - 2}} g}{}{L^2(\Sph^{n - 2})} \right)  \\
&\leq C \left( \frac{1}{\varepsilon} \int_{\partial \Omega \setminus F_\varepsilon}\langle\nu , x - y\rangle \, 
\e^{2 \tau \phi_2} |\partial_\nu( v_1 - w_2 )|^2 \, \dd A\right)^{\frac{1}{2}}  \\
&\qquad \times \left( \norm{g}{}{H^1(\Sph^{n - 2})} + \norm{\Delta_{\Sph^{n - 2}} g}{}{L^2(\Sph^{n - 2})} \right) 
\end{align*}
We now focus on the integral boundary term on the last inequality. Note that
\begin{equation}
(-\Delta + q_2) (v_1 - w_2) = - q|_{\Omega} v_1 \label{eq:SCHSchDIFF}
\end{equation}
with $ (v_1 - w_2)|_{\partial \Omega} = 0 $. We are going to use the Carleman estimate with boundary terms proved in \cite{KSU} 
\begin{prop} 
 Let   $q_2\in L^\infty(\Omega)$, there exist  $\tau_0>0$ and $C>0$  such that for all  $u\in \mathcal{C}^\infty(\Omega)$  , $u|_{\d\Omega}=0$  and $\tau>\tau_0$
\begin{multline*} C\tau^2\int_\Omega |\e^{-\tau \phi_2}u|^2 \, \dd x+ \tau \int_{\partial \Omega_+}\langle x-y,\nu(x)\rangle \e^{-\tau \phi_2}\partial_\nu u|^2
\, \dd A \\ \leq \int_\Omega |\e^{-\tau \phi_2}(\Delta -q)u |^2 \, \dd x- \tau \int_{\partial \Omega_-}\langle x-y,\nu(x)\rangle |\e^{-\tau \phi_2}\partial_\nu u|^2 \, \dd A. 
\end{multline*}
\end{prop}
Then  we obtain,
\begin{align*}
\bigg( \frac{1}{\varepsilon} \int_{\partial \Omega \setminus F_\varepsilon}\langle\nu , (x &- y) \rangle\,  \e^{2 \tau \phi_2} |\partial_\nu( v_1 - w_2 )|^2 \, \dd A  \bigg)^{1/2}  \\
&\leq C \left( \tau^{-1/2} \norm{\e^{\tau \phi_2} q v_1}{}{L^2(\Omega)} + \e^{c \tau} \norm{\partial_\nu( v_1 - w_2 )}{}{L^2(\partial \Omega_-(y))} \right) \nonumber \\
&\leq C \left( \tau^{-1/2} \norm{q}{}{L^\infty(\R^n)} + \e^{c \tau} \norm{\partial_\nu( v_1 - w_2 )}{}{L^2(\partial \Omega_-(y))} \right) 
\end{align*}
Here the constant depends additionally on $ \varepsilon $ and the distance of $ N $ to $ \ch(\Omega) $. We finally look at the $ L^2 $-norm on the boundary. Let $ \tilde{w}_2 \in H^1(\Omega) $ denote the solution of (\ref{eq:SCHschEQ}) with $ j = 2 $ and the following boundary condition $ \tilde{w}_2|_{\partial \Omega} = \chi \overline{\partial_\nu( v_1 - w_2 )|_{\partial \Omega}} \in H^{1/2} (\partial \Omega) $. Then, one has
\begin{align*}
\|\partial_\nu( v_1 &- w_2 )\|_{L^2(\partial \Omega_-(y))} \\ &= \left( \int_{\partial \Omega} \partial_\nu( v_1 - w_2 ) \tilde{w}_2 \, \dd A \right)^{1/2} \\
&= \left( \int_\Omega \Delta( v_1 - w_2 ) \tilde{w}_2 \, \dd x + \int_\Omega\langle \nabla( v_1 - w_2 ), \nabla\tilde{w}_2 \rangle\, \dd x \right)^{1/2} \\
&= \left( \int_\Omega q v_1 \tilde{w}_2 \, \dd x + \int_\Omega q_2 (v_1 - w_2) \tilde{w}_2 + \langle\nabla( v_1 - w_2 ) , \nabla\tilde{w}_2 \rangle\, \dd x \right)^{1/2} \\
&= \left( \int_\Omega q v_1 \tilde{w}_2 \, \dd x \right)^{1/2}.
\end{align*}
Notice that the square of the last term appears in (\ref{eq:SCHin-out}) if we change $ v_2 $ by $ \tilde{w}_2 $. Thus,
\begin{equation*}
\norm{\partial_\nu( v_1 - w_2 )}{2}{L^2(\partial \Omega_-(y))} = \duality{(\Lambda_{q_1} - \Lambda_{q_2})(v_1|_{\partial \Omega})}{\tilde{w}_2|_{\partial \Omega}}.
\end{equation*}
Taking into account that $ \mathrm{supp}\, \tilde{w}_2|_{\partial \Omega} \subset F $ one has
\begin{align*}
\|\partial_\nu( v_1 &- w_2 )\|_{L^2(\partial \Omega_-(y))} \\ &\leq \left( C \norm{\Lambda_{q_1} - \Lambda_{q_2}}{}{} \norm{v_1}{}{H^1(\Omega)} \norm{\tilde{w}_2}{}{H^1(\Omega)} \right)^{1/2} \\
&\leq \left( C \norm{\Lambda_{q_1} - \Lambda_{q_2}}{}{} \norm{v_1}{}{H^1(\Omega)} \norm{\partial_\nu( v_1 - w_2 )}{}{H^{1/2}(\partial \Omega)} \right)^{1/2} \\
&\leq C \norm{\Lambda_{q_1} - \Lambda_{q_2}}{1/2}{} \e^{c\tau} \left( \norm{ v_1 - w_2 }{}{L^2(\Omega)} + \norm{ \Delta(v_1 - w_2) }{}{L^2(\Omega)} \right)^{1/2}.
\end{align*}
In order to bound $ \norm{\partial_\nu( v_1 - w_2 )}{}{H^{1/2}(\partial \Omega)} $ in the last inequality, see \cite{BU}. Since $ v_1 - w_2 $ is solution of the problem (\ref{eq:SCHSchDIFF}) we can bound as follows
\begin{equation*}
\norm{\partial_\nu( v_1 - w_2 )}{}{L^2(\partial \Omega_-(y))} \leq C \norm{\Lambda_{q_1} - \Lambda_{q_2}}{1/2}{} \e^{c\tau} \norm{q}{}{L^\infty(\R^n)},
\end{equation*}
where the constant $ C $ depends   on $ \norm{q_j}{}{L^\infty(\Omega)} $ with $ j \in \{ 1, 2 \} $. Summing up,
\begin{gather*}
\int_{\partial \Omega \setminus F_\varepsilon} |\partial_\nu( v_1 - w_2 )| |v_2| \, \dd A \leq C \left( \norm{g}{}{H^1(\Sph^{n - 2})} + \norm{\Delta_{\Sph^{n - 2}} g}{}{L^2(\Sph^{n - 2})} \right) \\
\times \left( \tau^{-1/2} \norm{q}{}{L^\infty(\R^n)} +  \e^{c\tau} \norm{\Lambda_{q_1} - \Lambda_{q_2}}{1/2}{} \norm{q}{}{L^\infty(\R^n)} \right)^{1/2}.
\end{gather*}
Therefore, we finally get
\begin{multline}
\left| \int_{\tilde{\Omega}} q\, a_1 a_2 \, \dd x \right| \leq \frac{C}{\tau^{1/2}} \norm{q}{}{L^\infty(\R^n)} \left( \norm{g}{}{L^2(\Sph^{n - 2})} + \norm{\Delta_{\Sph^{n - 2}} g}{}{L^2(\Sph^{n - 2})} \right)  \\
+ C \left( \tau^{-1/2} \norm{q}{}{L^\infty(\R^n)} +  \e^{c\tau} \norm{\Lambda_{q_1} - \Lambda_{q_2}}{1/2}{} \norm{q}{}{L^\infty(\R^n)} \right)^{1/2} \label{es:boundaryMEAS} \\
\times \left( \norm{g}{}{H^1(\Sph^{n - 2})} + \norm{\Delta_{\Sph^{n - 2}} g}{}{L^2(\Sph^{n - 2})} \right)  \\ + C \norm{\Lambda_{q_1} - \Lambda_{q_2}}{}{} \e^{c \tau} \norm{g}{}{H^1(\Sph^{n - 2})}.
\end{multline}

In the following lines, we are going to relate the left hand side of the above estimate with the 2-plane transform of $ q $. Note that
\begin{equation*}
\int_{\tilde{\Omega}} q\, a_1 a_2 \, \dd x = 2^{-(n - 2)} \int_{\R \times \R \times S_\theta} q(y + t \theta + r \eta) g(\eta) \, \dd t \, \dd r \,\dd \sigma(\eta),
\end{equation*}
where $ S_\theta = \{ \eta \in \Sph^{n - 1} : \langle \eta,\theta \rangle= 0 \} $ and $ \dd \sigma $ is the volume form on $ S_\theta $ associated to euclidean metric restricted to $ S_\theta $. Note that given $ y \in N $ and $ \theta \in \Sigma $,
one has
\begin{multline*}
\left| \int_{S_\theta} Rq (y, \theta, \eta) g(\eta) \, \dd \sigma(\eta) \right| \leq C \left( \tau^{-1/4} \norm{q}{1/2}{L^\infty(\R^n)} +  \e^{c\tau} \norm{\Lambda_{q_1} - \Lambda_{q_2}}{1/4}{} \right) \\
\times \left( \norm{g}{}{H^1(\Sph^{n - 2})} + \norm{\Delta_{\Sph^{n - 2}} g}{}{L^2(\Sph^{n - 2})} \right),
\end{multline*}
for all $ y \in N $ and $ \theta \in \Sigma $. Obviously,
\[ \norm{Rq (y, \theta, \cdot)}{}{H^{-2}(S_\theta)} \leq C \left( \tau^{-1/4} \norm{q}{1/2}{L^\infty(\R^n)} +  \e^{c\tau} \norm{\Lambda_{q_1} - \Lambda_{q_2}}{1/4}{} \right) \]
for all $ y \in N $ and $ \theta \in \Sigma $. 
\end{proof}

We would like to have an expression   similar  to (\ref{2-plane1}), with a norm of the Radon transform $\Rad q$ in  the left hand side, suitable to apply  the stability result  in section 2.  This can be  achieved in the  three dimensional case, since  the 2-plane  transform $ Rq (y, \theta, \eta)$, the integral  of $ q $ in the plane $ P = \{ y \} + [ \theta, \eta ] $, is  a reparametrisation of the Radon transform. 

 From now on, we will restrict our analysis to dimension $ n = 3 $.
As we pointed out before, the variable $y$ in (\ref{2-plane1}) is  redundant,  since    the natural  coordinates of the 2-plane transform are $(y, P)$ where $P= [ \theta, \eta ] $ and $y\in P ^\perp$. The variable $y$ in (\ref{2-plane1}) can not be restricted to $P^\perp$, due to the $\eta$-integral involved in the partial Sobolev norm.

  We have  to make  sense    of the integral in (\ref{2-plane1}) in the  appropriate coordinates  of the Radon transform given by $(s,\omega)\in \R \times \mathbb{S}^2$. We can write poinwise  $Rq(y, \theta, \eta)= \Rad q(s,\omega ) $, where $s= \langle y, \omega\rangle $ and
  $\omega= \theta\times\eta$ (vector product).  
   To simplify  the notation, we will assume that $\Sigma=\Sigma_{4\delta}$ is bounded  by the two planes parallel to $H$ at distance $4\delta $ of the origin. 
   
  We introduce geodesic polar coordinates on $\mathbb{S}^2$ in the following way: given $\omega_0\in \mathbb{S}^2$ we can find $\theta_0\in \Sigma_{\delta}$  and $\eta_0$ so that they form an orthonormal frame.

  We will denote 
  $\eta_0=e_1$, $\theta_0=e_2$ and $\omega_0=e_3$.
  Let   
  $\theta = \theta (\phi)=\cos \phi e_2+ \sin\phi e_1$ with $|\phi| <\delta$. 
   Let $\eta$ run 
  the  geodesic $S_\theta$  according to $\psi= dist(\eta, e_3)$,  we have 
  $\eta=\eta(\phi, \psi)= \cos\phi \sin \psi e_1- \sin \phi\sin\psi e_2+ \cos\psi e_3$.  We will perform the  following calculations on the time  zone 
  $$\Gamma= \Gamma_{\omega_0, \theta_0} (\delta) = 
  \{\eta(\phi, \psi): | \phi|<\delta,  0<|\psi|<\pi\}.$$
  We will need to write $Rq(y, \theta, \eta)$ in the standard coordinates 
  $\Rad q (s, \omega(\phi, \psi))$, where $\omega(\phi, \psi)=\theta(\phi)\times\eta(\phi, \psi)$. For a fix $\phi$ the map $\eta\to \omega$ is  just  a rotation of angle $\pi/2$ on the geodesic $S_\theta$, hence
  $\omega(\phi, \psi)= \eta(\phi, \psi-  \pi/2)$. As a map from  $\Gamma$ to $\Gamma$ this rotation collapses  the segments 
  $\psi= \pm \pi/2$ to the points  $\pm \omega_0$. We then  restrict the time  zone to 
  \begin{equation}\label{tz}\tilde \Gamma  =\tilde \Gamma_{\omega_0, \theta_0} (\delta)=
  \{\eta(\phi, \psi): | \phi|<\delta, \psi \in J\},
  \end{equation}
  where $J=\{ \delta/8<|\psi|<\pi/2-\delta/8\}\cup\{\pi/2+ \delta/8<|\psi|<\pi-\delta/8\}$.
  
  In these coordinates  we have:
  
  \begin{lem}(Local estimate, $n=3$)  Let $B(0,\alpha)\subset N$. Then
 \begin{multline}
    \int _{ B(0, \alpha)}\int _{-\delta}^\delta \left( \int_J |Rq(y, \theta, \eta)|^{6/5} \, \dd\psi\right)^{5/6}\, \dd\phi \,\dd y \leq \\
   C \left( \tau^{-1/4} \norm{q}{1/2}{L^\infty(\R^n)} +  \e^{c\tau} \norm{\Lambda_{q_1} - \Lambda_{q_2}}{1/4}{} \right)^{1/3},\label{X-ray}
   \end{multline} 
   where $C$ only depends on $\delta$  and $\alpha$ .
  \end{lem}
 
\begin{proof}
We  will obtain (\ref{X-ray})  by interpolation of the following estimates.
 \begin{multline}\label{X-ray1}
  \int _{ B(0, \alpha)}\int _{-\delta}^\delta   \| Rq(y, \theta(\phi), \eta( \phi, \cdot))\|_{H^{-2}(J)}\, \dd\phi \, \dd y \leq \\
   C \left( \tau^{-1/4} \norm{q}{1/2}{L^\infty(\R^n)} +  \e^{c\tau} \norm{\Lambda_{q_1} - \Lambda_{q_2}}{1/4}{} \right).
   \end{multline}   
   and
   \begin{multline}\label{X-ray2}
\int _{ B(0, \alpha)}\int _{-\delta}^\delta \int_J  ( |\partial_\psi Rq(y, \theta(\phi), \eta( \phi, \psi))|\\ +
    |Rq(y, \theta(\phi), \eta( \phi, \psi))|  ) \, \dd \psi \, \dd\phi \, \dd y  \leq  
    C\|q\|_{L^2}.
\end{multline} 
Estimate (\ref{X-ray1}) follows  easily from  (\ref{2-plane1}).
To prove estimate (\ref{X-ray2}) we start with the derivative term, we change the order of integration  and write $ y=s\theta \times \eta + y'$ with $y' \in [ \theta, \eta ]$, 
 \begin{multline*}
 \int _{ B(0, \alpha)}\int _{-\delta}^\delta \int_J  |\partial_\psi Rq(y, \theta(\phi), \eta( \phi, \psi))| \, \dd\psi \, \dd\phi \,  \dd y \\ \leq 
 \int _{-\delta}^\delta \int_J \int_{-\alpha}^{\alpha}\int _{B_\omega}|\partial_\psi Rq(s\omega+y', \theta(\phi), \eta( \phi, \psi))| \dd y' \, \dd s
\, \dd\psi \, \dd\phi,
\end{multline*}
where $\omega=\omega(\phi,\psi)$ and $B_\omega= B(0,\alpha)\cap[ \theta, \eta ] $.

Since $y'$ does not change  the  2-plane X-ray transform, we write  the  above as
$$C\alpha^2\int _{-\delta}^\delta \int_J \int_{-\alpha}^{\alpha} |\partial_\psi Rq(s\omega, \theta(\phi), \eta( \phi, \psi))| \,\dd s \, \dd\psi \, \dd\phi,$$
 This  expression can be  written in terms of the Radon transform as
\begin{multline} C\alpha^2\int _{-\delta}^\delta \int_J \int_{-\alpha}^{\alpha} |\partial_\psi \Rad q(s,  \omega(\phi , \psi))| \,\dd s \, \dd\psi\, \dd\phi
\\ \label{radon3d}
=C\alpha^2\int _{-\delta}^\delta \int_J \int_{-\alpha}^{\alpha} |\partial_\psi \Rad q(s,  \eta(\phi , \psi-\pi/2))| \, \dd s \, \dd\psi \, \dd\phi.
\end{multline}
 To simplify the notation, following \cite{Nat}, we  will use the homogeneous of  degree $-1$ extension  of $\Rad(s, \eta)$ to $\eta\in \R^3$, to write 
 $$\partial_\psi \Rad q(s,  \eta(\phi, \psi-\pi/2))=\nabla_\eta\Rad q(s, \eta)\cdot \frac{\partial\eta}{\partial \psi}= \frac{\partial}{\partial s}\Rad(  xq)(s, \eta)\cdot \frac{\partial\eta}{\partial \psi},
 $$ 
 then (\ref{radon3d}) can be  bounded by
 \begin{equation}
 C\alpha^2  \int_{-\alpha}^{\alpha}\int _{-\delta}^\delta  \int_J \bigg|\frac{\partial}{\partial s}\Rad(xq)(s,  \eta(\phi , \psi-\pi/2))\bigg| \, \dd\psi \, \dd\phi \, \dd s.
\end{equation}
From  the fact that on $\tilde\Gamma$ one has
$  \delta/8\leq|\frac{d\sigma(\eta)}{d\phi d\psi}|= |\cos \psi |\leq 1-\delta/8$, we obtain that this can be bounded by
\begin{align*}
&\leq\frac{ \alpha^2}{\delta} \int_{-\alpha}^{\alpha}\int _{\tilde \Gamma} \bigg|\frac{\partial}{\partial s}\Rad(xq)(s,  \eta )\bigg| \,\dd\sigma(\eta) \, \dd s.
\intertext{By using Cauchy-Schwarz inequality, this is majorized by}  
&\leq \frac{\alpha^{5/2}}{\delta^{1/2}}\left( \int_{-\alpha}^{\alpha}\int _{\tilde \Gamma} \bigg|\frac{\partial}{\partial s}\Rad(xq)(s,  \eta )\bigg|^2 \dd\sigma(\eta)
\,\dd s\right)^{1/2} \\
&\leq C \|\Rad(xq)\|_{H^1(\mathbb{R}\times\Sph^2  )}.
\end{align*}
Since in dimension $n=3$, we have 
$$\|\Rad(xq)\|_{H^1(\mathbb{R}\times \Sph^2)}\leq C\|(1+|x|)q\|_{L^2},$$ 
this fact, together with  an easier estimate for the zero order term in  (\ref{X-ray2}) give, 
$$\int_{B(0,\alpha)}\int_{\theta \in (-\delta, \delta)}\int _{S_\theta}|
\partial_ \eta Rq(y\cdot \theta \times \eta, \theta \times\eta)| \, \dd\eta \,\dd\theta \, \dd y\leq   C( \||x|q\|_{L^2} + \| q\|_{L^2}),$$
this ends the proof   (\ref{X-ray2}) and the Lemma.
 \end{proof}
 Now we return to the standard coordinates.
 
 \begin{clry}\label{radon3} (n=3) 
 In the conditions of proposition \ref{twoplanes}, assume that $B(0,\alpha)\subset N$. Then 
  \begin{multline} \label{radonDirichlet-to-Neumann}
  \int _{ -\alpha/2}^{ \alpha/2} \int_{\Sph^2}|\Rad q (s,\omega)| \, \dd\sigma(\omega) \,\dd s \\ \leq C \left( \tau^{-1/4} \norm{q}{1/2}{L^\infty(\R^n)} 
  +  \e^{c\tau} \norm{\Lambda_{q_1} - \Lambda_{q_2}}{1/4}{} \right)^{\frac{1}{3}},
  \end{multline}
  where $\Rad$ denotes  the Radon transform.
  \end{clry}  
  \begin{proof}
Let us take the open covering of $\Sph^2$ given by 
$\{\tilde\Gamma_a(\delta)\}_{ a\in A}$, see\eqref{tz}, where $A= \{(\omega_0, \theta_0)\in \Sph^2\times\Sigma_\delta: \langle\omega_0, \theta_0\rangle=0\}$. By taking a finite subcovering and  a partition of unity subordinated to it, we might reduce to prove  the statement locally for the  sets $\tilde\Gamma_a(\delta)\subset \Sph^2$. 
Then, for the local coordinates given in the lemma, $\omega=\eta(\phi, \psi-\pi/2)$, we have 
\begin{multline}\label{rela}
\int _{ -\alpha/2}^{ \alpha/2} \int_{\tilde\Gamma_a}|\Rad q (s,\omega)| \, \dd\sigma(\omega) \, \dd s 
\\ \leq \frac{C}{\delta}\int _{ -\alpha/2}^{ \alpha/2}\int_{-\delta}^\delta\int_J | \Rad q (s,\eta(\phi,\psi))| \, \dd\phi \, \dd\psi \, \dd s.
\end{multline}
For any $y\in \R^3$ such that $\langle y,\eta(\phi,\psi)\rangle= s$  we have in terms of the  two-plane transform
$$\Rad q (s,\eta(\phi,\psi))= R q(y, \theta(\phi),\eta(\phi, \psi)). $$
By taking $y\in B(0, \alpha/2)$, so that  $y= s\eta+ y'$ with $y'\in \eta^\perp$ , and  denoting $B_\alpha= B(0, \alpha/2)\cap\eta^\perp$,  we can write 
$$\Rad q (s,\eta(\phi,\psi))=|B_\alpha |^{-1}\int_{B_\alpha}|R q( s\eta+ y', \theta ,\eta )| \, \dd y'.$$
Inserting this in (\ref{rela}), we have
\begin{align*}
\int _{ -\alpha/2}^{ \alpha/2} \int_{\tilde\Gamma_a}&|\Rad q (s,\omega)| \, \dd\sigma(\omega) \, \dd s \\
&\leq C(\alpha, \delta)\int _{ -\alpha/2}^{ \alpha/2}\int_{-\delta}^\delta\int_J \int_{B_\alpha}|R q( s\eta+ y', \theta ,\eta )| 
\, \dd y' \, \dd\phi  \, \dd\psi \, \dd s \\
&\leq C(\alpha, \delta)\int_{B(0,\alpha)}\int_{-\delta}^\delta\int_J |R(y, \theta, \eta)| \, \dd\psi \, \dd\phi \, \dd y.
\end{align*}
The Corollary follows from the local estimate, together with  H\"older's inequality.
\end{proof}

\subsection{End of the proof  of Theorem \ref{KSUstability}}
We want to use the   Corollary \ref{radon3}  together  with Theorem \ref{th:localLOGstability}.  Assume  $x_0\in P$, the convex penumbra     from $N$, and let $y\in N$, so that $\langle x_0-y, \nu(x_0)\rangle=0$. To  simplify  notation we assume $y=0$ and  take $\alpha >0$, such that $B(0,\alpha)\subset N$. 

Assume $x_0\in \supp q$ (otherwise there is nothing to estimate) then the plane $H_0(0, \nu(x_0))$, see \eqref{hiperp}, is  a supporting plane of $\supp q$, as required  by (b) of  Theorem \ref{th:localLOGstability} and,  from the Corollary,  we have estimate \eqref{radonDirichlet-to-Neumann}.
We need, as required  by Theorem \ref{th:localLOGstability}, to use  the Radon transform $\Rad_{x_0}$ with respect to the  affine reference with center  at the point  $x_0\in P$.
Consider  $\beta= \frac{\alpha }{4 \delta}$, where 
$\delta= \sup\{ |z-y|: z\in \Omega, y \in N\}$.
Then  the set of planes $\{H_0(s,\omega):  |s|<\alpha/2\, ,  \omega \in S^2\}$ contains  the  set $\{H_{x_0}(s,\omega):  |s|<\alpha/8\, ,  \omega \in \Gamma\}$ where $\Gamma=\Gamma(\nu(x_0), \beta)$ is  defined in \eqref{gamma}.

This  allows us to write for $I=(-\alpha/8, \alpha/8)$,
\begin{gather*}
\int _{ I}(1+|s-\langle \omega,\zeta-y_0 \rangle|) \int_{\Gamma}|\Rad_{x_0} q (s,\omega)| \,\dd\sigma(\omega) \, \dd s \\
\leq C \int _{ -\alpha/2}^{ \alpha/2} \int_{\Sph^2}|\Rad q (s,\omega)| \, \dd\sigma(\omega) \, \dd s \\
\leq C \left( \tau^{-\frac{1}{4}} \norm{q}{1/2}{L^\infty(\R^n)} +  \e^{c\tau} \norm{\Lambda_{q_1} - \Lambda_{q_2}}{1/4}{} \right)^{\frac{1}{3}}.
\end{gather*}
Finally, the choice  $\tau=\frac{1}{8c}|\log \|\Lambda_{q_1}-\Lambda_{q_2}\|_{B\to F}| $, together with  Theorem \ref{th:localLOGstability} gives the estimate in Theorem \ref{KSUstability} for $G$ a neighborhood of $x_0\in P$. The claim  for  $G$  a neighborhood  of $P$ follows  by standard  arguments.

\end{section}

\end{document}